\documentclass[reqno,11pt]{amsart}
\usepackage[margin=1in]{geometry}
\newcommand{\vvNumberWithin}{subsection}

\RequirePackage{amsthm}
\RequirePackage{amsmath}
\RequirePackage{amsfonts}
\RequirePackage{amssymb}
\RequirePackage{mathtools,amsfonts}
\usepackage[all]{xy}
\RequirePackage{tikz-cd}
\RequirePackage{float}

\RequirePackage{txfonts}
\RequirePackage{lmodern}
\RequirePackage{enumerate}
\RequirePackage[shortlabels]{enumitem}
\RequirePackage{xstring}

\newcommand*{\FIXME}[1]{}

\newcommand*{\Q}{\mathbb{Q}}

\newcommand*{\Z}{\mathbb{Z}}

\DeclareMathOperator{\Spec}{Spec}
\DeclareMathOperator{\Hom}{Hom}
\let\hom\relax
\DeclareMathOperator{\hom}{Hom}
\DeclareMathOperator{\End}{End}
\DeclareMathOperator{\cone}{Cone}
\DeclareMathOperator{\For}{For}

\newcommand{\VVC}[1]{{}}

\newtheoremstyle{misc}%
     {\topsep}
     {\topsep}
     {}
     {}
     {\itshape}
     {}
     { }
     {}

\newtheoremstyle{newdef}{\medskipamount}{\medskipamount}{}{}{\bfseries}{.}{ }{{\normalfont{\thmnumber{#2}.}}\thmnote{ #3}}

\newtheorem{master}{Master}[\vvNumberWithin]
\theoremstyle{newdef}
\newtheorem{npara}[master]{NamedParagraph}

\swapnumbers
\theoremstyle{plain}
\newtheorem{theorem}[master]{Theorem}
\newtheorem*{theorem*}{Theorem}

\newtheorem*{result*}{Result}
\newtheorem{lemma}[master]{Lemma}
\newtheorem*{lemma*}{Lemma}
\newtheorem{corollary}[master]{Corollary}
\newtheorem*{corollary*}{Corollary}
\newtheorem{proposition}[master]{Proposition}
\newtheorem*{proposition*}{Proposition}
\newtheorem{assumption-proposition}[master]{Assumption/Proposition}

\theoremstyle{definition}
\newtheorem{example}[master]{Example}
\newtheorem*{example*}{Example}

\newtheorem*{application*}{Application}
\newtheorem{definition}[master]{Definition}
\newtheorem*{definition*}{Definition}
\newtheorem{remark}[master]{Remark}
\newtheorem*{remark*}{Remark}
\newtheorem{para}[master]{}
\newtheorem*{para*}{}
\newtheorem{notation}[master]{Notation}
\newtheorem*{notation*}{Notation}

\newtheorem*{question*}{Question}

\newtheorem*{problem*}{Problem}

\theoremstyle{remark}

\newtheorem*{exercise*}{Exercise}

\makeatletter
\@addtoreset{master}{section}
\makeatother

\makeatletter
\@addtoreset{equation}{section}
\makeatother
\numberwithin{equation}{subsection}
 \renewcommand{\theequation}{%
      \ifnum\value{subsection} > 0
      \ifnum\value{subsubsection} > 0
      \thesubsubsection.\Alph{equation}%
      \else%
      \thesubsection.\Alph{equation}%
      \fi%
      \else%
      \thesection.\Alph{equation}%
      \fi%
    }
    \usepackage{hyperref}
\newcommand{\nocontentsline}[3]{}
\newcommand{\tocless}[2]{\let\tempcontentsline=\addcontentsline\let\addcontentsline=\nocontentsline#1{#2}\hspace{-1em}\let\addcontentsline=\tempcontentsline}

\newcommand*{\vvspan}[1]{{\langle #1 \rangle}}

\newcommand{\cross}{\times}
\newcommand{\tensor}{\otimes}

\title{Intersection complex of any threefold as a Chow motive}
\author{Shruti~Rastogi}
\email{mp18021@iisermohali.ac.in}
\address{Department of Mathematical Sciences, Indian Institute of Science Education and Research (IISER) Mohali, Knowledge City, Sector 81, Mohali, Punjab - 140306, India.}
\author{Vaibhav~Vaish}
\email{vaibhav@iisermohali.ac.in}
\address{Department of Mathematical Sciences, Indian Institute of Science Education and Research (IISER) Mohali, Knowledge City, Sector 81, Mohali, Punjab - 140306, India.}

\begin{document}
\begin{abstract} 
Motivated by the characterization of the intersection complex in terms of S.Morel’s weight truncations, we introduced an object $EM^{F}_{X}$ in the setting of motivic sheaves for certain schemes $X$ and weight profiles $F$. In this article, we show that when $X$ is \emph{any} threefold, this object satisfies Wildeshaus’s characterization of a motivic intersection complex. In particular, we demonstrate that the construction is a suitably functorial Chow motive lifting the motivic intersection complex for an arbitrary threefold. 
\end{abstract}

\maketitle
\tableofcontents

\addtocontents{toc}{\protect\setcounter{tocdepth}{1}}
\section{Introduction}

\subsection{} In the 1960s Grothendieck envisioned the theory of motives as a unifying framework for various (Weil) cohomology theories. As further enhanced and articulated by Beilinson \cite{beilinson1987height} we can even conceive a relative version -- there should exist a certain abelian category $\mathcal{M}M(X)$ of mixed motivic sheaves over $X$, together with realization functors to each Weil cohomology theory. Furthermore, these categories should reflect properties of the Weil cohomology theories, in particular they should come with a formalism of Grothendieck's six operations and weights.

One pathway to constructing the category of mixed motives is to first construct its derived category $DM(X)$, and then construct a suitable $t$-structure on the same whose heart would be the conjectural $\mathcal MM(X)$ -- a proposal that is credited to Deligne. This conjectural $t$-structure corresponds to the perverse $t$-structure \cite{BBD} in the setting of realizations.

Due to work of M. Hanamura \cite{hanamura, hanamura2}, M. Levine \cite{levine} and F.Morel-Voevodsky \cite{voevodsky_morel_99}, Jardine \cite{jardine_2000}, Ayoub \cite{ayoub_thesis_1, ayoub_thesis_2}, and Cisinski-Deglise \cite{Cisinski_2019} we now have candidates for the category $DM(X, \mathbb Q)$ (here $\mathbb Q$ emphasizes the coefficients -- we will work only with rational coefficients in this article) -- for our purposes it can be taken to be the category of etale motivic sheaves without transfers $DA(X, \mathbb Q)$ of Ayoub\cite[2.1]{ayoub2012relative} or, alternatively, the Beilinson motives, $DM_{B,c}(X, \mathbb Q)$ of Cisinski-Deglise \cite{Cisinski_2019}. This formalism has several of the expected properties: it is equipped with the formalism of Grothendieck's six functors and has the correct relation with the cycle class groups (by construction), has weights (see Bondarko \cite{bondarko_weights_2010} and Hebert \cite{hebert2011structure}) and is equipped with several expected realization functors (see \cite{ayoub2012relative}, \cite{MR1775312}, \cite{ivorra2007realisation}). However, the motivic $t$-structure remains an elusive dream.

\subsection{} The heart of the motivic t-structure is expected to be a compactly generated abelian category, with compact objects that are both Noetherian and Artinian. So, another approach to study the conjectural abelian category of mixed motives is via its simple objects. These simple objects are expected to realize to intersection complexes $IC(Y, \mathcal{L})$, where $Y \subset X$ is closed and $\mathcal{L}$ is a local system defined generically on $Y$ which can even be taken to be in mixed categories. 

A natural and interesting question is whether these intersection complexes $IC(Y, \mathcal{L})$ admit lifts to the motivic category $DM(X, \mathbb{Q})$. These expected lifts are known as motivic intersection complexes and have been intrinsically characterized by Wildeshaus in \cite{wildeshaus_ic}. When such a lift exists, it is canonical.  In later work \cite{wildeshaus_shimura_2012}, Wildeshaus also introduced a slightly weaker notion, in which the motive is unique only up to an isomorphism. Although this version is weaker than the canonical one, it has the advantage of existing unconditionally in several important cases - for instance, for
the Baily-Borel compactification of Shimura varieties.

In \cite{vaish2017weight}, the second-named author constructed a candidate for the lift of intersection complex $IC(X,\mathbb Q)$, the so called motivic intersection complex, for an arbitrary threefold in $DM(X,\mathbb Q)$ which was denoted $EM_X^F$.  However, the structural properties of the object (in particular, whether it is a Chow motive) were left open. In \cite{vaish2018motivic}, he further reconciled the constructions with those in \cite{wildeshaus_shimura_2012} and demonstrated that in case the variety is (the Bailey-Borel compactification of) a Shimura variety, then the motive so constructed is the correct object for the motivic intersection complex and in particular a Chow motive. This left open the question of what happens for threefolds that are not Shimura varieties. 

This article aims to address this gap and, using different techniques, we show that for \emph{any} threefold $X$, the candidate $EM^{F}_{X}$ is indeed the motivic intersection complex, and in particular a Chow motive. Our main result is as follows:
\begin{theorem}\label{mainthm intro}(see \ref{main theorem})
Let $X$ be an irreducible variety of dimension $3$ over a field $k$ of characteristic $0$. Then the following hold: 
  \begin{enumerate}[(i)]
  \item The object $EM^{F}_{X} \in DM^{coh}_{3,dom}(X)\subset DM(X)$ is of weight $0$ in the sense of Bondarko. In particular, it is a Chow motive.
  \item For any open immersion $j:U \hookrightarrow X$ with $U$ regular, $j^{*}EM^{F}_{X} \cong 1_{U}$ and the map induced by pullback $j^{*}$ 
  \[ End(EM^{F}_{X}) \rightarrow End(1_{U})
  \] is an isomorphism
    \end{enumerate}
    In particular, the construction $EM_X^F$ satisfies Wildeshaus' (stronger) criteria \cite{wildeshaus_ic} for being a motivic intersection complex.
\end{theorem}
\subsection{} Now we briefly motivate the method used to prove Theorem \ref{mainthm intro}. The two key ingredients used are weight conservativity (\cite{wildeshaus_conservativity})- a version of which we prove (\ref{weightconservativity_our}) -  and a generalization of the refined Chow–Künneth decomposition of a surface as in \cite{murre_kahn_pedrini}.

We want to show that the weight of $EM^{F}_{X}$ is $0$. Firstly, recall that by the punctual gluing of the weight structures \cite{vaish2017punctual}, we can reduce the problem to showing that $s^{*}(EM^{F}_{X})$
has weight $\leq 0$ and $s^{!}(EM^{F}_{X})$ has weight $\geq 0$ for all points $s : \Spec k(s) \hookrightarrow X$.

Given any resolution $\pi:X' \rightarrow X$ of a threefold $X$, the fibers along $\pi$ have dimension 2 over a zero-dimensional stratum and on the complement they have dimension $\leq 1$. For a point $s$ such that the fiber over $s$ is of dimension $\leq 1$, the objects $s^{*}EM^{F}_{X}$ and $s^{!}EM^{F}_{X}$ lie in the triangulated category generated by the motives of smooth projective curves (upto Tate twists), and in particular in the category $DM^{Ab}(k(s)) \subset DM(k(s))$ generated by the motives of abelian varieties (up to Tate twists). Then the conservativity of the restriction of the realization functor to $DM^{Ab}(k(s))$ allows us to calculate the weights of $s^{*}EM^{F}_{X}$ as well as $s^{!}EM^{F}_{X}$ in the realizations.

This leaves us with the tricky case where the fiber at $s$ is a surface $Y$ (which we can assume to be smooth for the purpose of this introduction). Here, phantom motives may appear as the summands of $h^{2}(Y)$ and as such the weight conservativity fails.
\subsection{} To address this situation, we define a refinement of S.~Morel's $t$-structure -- recall that for each integral weight $d$, S.~Morel constructs a $t$-structure: $({^wD^{\le d}}(k),{^wD^{> d}}(k))$ on the mixed realization category, which we denote as $D(k)$ (say $D^b(k,\mathbb Q_l)$ for $k$ finite, or $D^bMHS(k)$, the derived category of mixed Hodge structures for $k\subset \mathbb C$). These $t$-structures have a trivial heart, and satisfy the relation: 
\begin{align*}
    {^wD^{\le d}}(k) \subset {^wD^{\le d+1}}(k)& &{^wD^{> d}}(k)\supset {^wD^{> d+1}}(k)
\end{align*}
It can be shown that in the presence of the motivic $t$-structure, these would lift to $t$-structures, denoted  $({^wDM^{\le d}},{^wDM^{> d}})$, on $DM(k)$. We propose a refined $t$-structure, $({^wDM^{\le 1+alg}}, {^wDM^{> 1+alg}})$ with the expectation
\begin{align*}
    {^wDM^{\le 1}}(k) \subset{^wDM^{\le 1+alg}}(k) \subset {^wDM^{\le 2}}(k)& &{^wDM^{> 1}}(k)\supset  {^wDM^{> 1+alg}}(k)\supset {^wD^{> 2}}(k)
\end{align*}
and satisfying some useful properties. Note that while $({^wDM^{\le d}},{^wDM^{> d}})$ has not been constructed on $DM(k)$, it has been on suitable sub-categories. E.g., for $d=0,1$, this has been constructed on the subcategory $DM^{coh}(k)$ (cohomologically) generated by effective Chow motives \cite{vaish2017punctual} while for arbitrary $d$ this has been constructed on the subcategory $DM^{coh}_{2}(k)$ (cohomologically) generated by surfaces with arbitrary Tate twists \cite{vaish2018motivic}.

Due to the lack of existence of the motivic $t$-structure on $DM(k)$, the characterization of the $t$-structure $({^wDM^{\le 1+alg}}, {^wDM^{> 1+alg}})$ becomes somewhat technical and in the introduction, we only layout the following properties:

\begin{theorem}\label{B-int}(see \ref{1+alg for DM-coh(k)} and \ref{1+alg for surfaces} for details) There is a $t$-structure on $DM^{coh}(k)$ which we denote by $(^{w}DM^{\leq 1+alg}(k), \;^{w}DM^{>1+alg}(k))$, such that if $w_{\le 1+alg}$ and $w_{>1+alg}$ denote the corresponding truncation functors, we have:
\begin{enumerate}[(a)]
    \item For any smooth proper surface $Y$ over $k$, the natural triangle 
    \[
       h^2_{alg} \cong w_{\le 1+alg}(h^2(Y))\rightarrow h^2(Y) \rightarrow w_{> 1+alg}(h^2(Y)) \cong h^2_{tr}(Y) \rightarrow
    \]
    is a split triangle of Chow motives which realizes the algebraic and the transcendental part of $h^2(Y)$ in the refined Chow–Künneth decomposition of \cite{murre_kahn_pedrini}.
    \item ${^wDM^{\le 1}}(k) \subset {^wDM^{\le 1+alg}} \subset DM^{Ab}(k)$.
 \end{enumerate}
\end{theorem}

\subsection{} Returning to the outline of the proof, we are interested in the weights of the objects $s^*EM^F_X$ and $s^!EM^F_X$, where $s$ is a point in $X$ over which the fiber is two-dimensional. For computing the weights of $s^*EM^F_X$, it is enough to compute the weights of $w_{\le 1+alg}s^*EM^F_X$ and those of $w_{> 1+alg}s^*EM^F_X$. The former, lying in $DM^{Ab}(k)$, is amenable to weight conservativity \cite{wildeshaus_conservativity}. However, the application is not as straightforward - this is because of the absence of mixed realizations of the triangulated category of mixed motivic sheaves that commute with the six functors. To overcome this difficulty, we first observe that on the open complement $j:U\hookrightarrow X$ of points over which the fibres are 2 dimensional, the restriction $EM^{F}_{X}|_{U} \cong EM^{F}_{U}$ is a summand of $\pi_{*}1_{ X'}|_{U}$, where  $\pi: X' \rightarrow X$ is a resolution of singularity. Therefore, it is enough to compute the weights of $s^{*}j_{*}\pi_{*}1_{X'}|_{U}$. This in turn is related to the cone of the homological map $h(Z)\rightarrow h(X')$ and the cohomological map $h^{coh}(X') \rightarrow h^{coh}(Z)$ where $Z\subset X' \rightarrow X$ denotes the fiber over $s$. Then the realization of a cohomological map in Huber's mixed category $D_{\mathcal{MR}}$ (\cite{MR1439046},\cite{MR1775312}), and of homological map in the bounded derived category of mixed Hodge structures $D^{b}(MHS(k))$, yields the expected map in the respective realizations. Since the Betti realization functor from $D^{b}(MHS_{\mathbb{Q}}(k)) \rightarrow D^{b}_{c}(k)$ is faithful, we are able to prove the following:
\begin{proposition}(see \ref{realizations of EM at each point})
    Let $k$ be a field of characteristic zero with a fixed embedding $\sigma: k \hookrightarrow \mathbb{C}$. Let $D^{b}(MHS_{\mathbb{Q}}(k))$ denote the bounded derived category of mixed Hodge structures over $k$ and $R_{\mathcal{M}}: DM(k) \rightarrow D^{b}(MHS_{\mathbb{Q}}(k))$ denote the realization functor. Then for any point $s$ of $X$ we have
    \[ R_{\mathcal{M}}(s^{*}EM^{F}_{X}) \cong s^{*}IC_{X}.\]
    \end{proposition}
    \subsection{}Having done so, we are now in a position to apply weight conservativity. Our goal is to determine the weight of $R_{\mathcal{M}}(w_{\leq 1+alg}s^{*}EM^{F}_{X})$. For this purpose, it is both useful and possible to construct an analogue of the refined Morel's t-structure in the mixed category; see Proposition \ref{mixed realisation commute with 1+alg} for details. We denote the corresponding truncation functors also by $(w_{\leq 1+alg},w_{>1+alg})$. Since $R_{\mathcal{M}}$ commutes with $w_{\leq 1+alg}$ and $s^{*}IC_{X}$ has weight $\leq 0$, we obtain $R_{\mathcal{M}}(w_{\leq 1+alg}s^{*}EM^{F}_{X})$ has weight $\leq 0$. Then weight conservativity implies that  $w_{\leq 1+alg} s^{*}EM^{F}_{X}$ is of weight $\leq 0$, as desired.

Now for $w_{> 1+alg}s^*EM^F_X$, the object may involve phantom motives, and it is not useful to go to the realizations. However, we get lucky, and we are able to use the geometry of the situation to explicitly compute the weights and show that it is of weight $\le 0$, as required. 

The case for $s^!EM^F_Y$ is similar, except that we have to use a (Tate) twisted version of the $t$-structure as described above, thus completing the proof of (i).
\subsection{}The proof of part $(ii)$ is fairly straightforward. We begin by recalling how the object $EM^{F}_{X}$ was defined. In \cite{vaish2017weight}, Vaish constructed Morel's t-structure in the relative situation using the punctual gluing of t-structures. In particular, it was shown that for certain monotone step functions $F:\{0,\cdots,\dim X\} \rightarrow \mathbb{Z}$, there exists a t-structure, with truncation functors denoted as $(w_{\leq F}, w_{>F})$, on appropriate subcategories of $DM(X)$, whose realization corresponds to Morel's t-structure associated with $F$ in the category of mixed sheaves. In our case, where $X$ is a threefold and $F=\{3 \mapsto 3, 2 \mapsto 3, 1 \mapsto 2, 0 \mapsto 2\}$, this construction yields a t-structure on the subcategory $DM^{coh}_{3,dom}(X) \subset DM^{coh}(X) \subset DM(X)$. Here $DM^{coh}(X)$ is the category of cohomological motives and the subcategory $DM^{coh}_{3,dom}(X)$ is obtained by imposing restrictions on the dimensions of fibres.

Now, given the motivic Morel's t-structure the object $EM^{F}_{X} \in DM(X)$ is defined as follows: for any open dense immersion $j: U \hookrightarrow X$ with $U$ regular, \[EM^{F}_{X}:= w_{\leq F}j_{*}1_{U}.\]
Since $j^{*}$ is weight exact, $j^{*}$ commutes with $w_{\leq F}$ and therefore, $j^{*}EM^{F}_{X} \cong 1_{U}$.

To show the isomorphism of endomorphism rings,  we consider the decomposition triangle of $j_{*}1_{U}$:
\[  w_{\leq F}j_{*}1_{U} \rightarrow j_{*}1_{U} \rightarrow w_{>F}j_{*}1_{U} \rightarrow\]
The functor $\Hom(w_{\leq F}j_{*}1_{U},-)$ give rise to a long exact sequence
\begin{align*} \cdots \rightarrow \Hom(w_{\leq F}j_{*}1_{U}, w_{>F}j_{*}1_{U}[-1])\rightarrow\Hom(w_{\leq F}j_{*}1_{U},w_{\leq F}j_{*}1_{U}) \rightarrow &\Hom(w_{\leq F}j_{*}1_{U}, j_{*}1_{U}) \rightarrow \\
&\Hom(w_{\leq F}j_{*}1_{U},w_{>F}j_{*}1_{U}) \rightarrow \cdots 
\end{align*}
Due to the orthogonality of Morel's t-structure, the first and last terms are zero. Then by adjunction, we obtain the isomorphism as required:
\[\End(EM^{F}_{X}) =  \Hom(w_{\leq F}j_{*}1_{U},w_{\leq F}j_{*}1_{U})\cong \Hom(w_{\leq F}j_{*}1_{U},j_{*}1_{U}) \cong \End(1_{U}) .\]

\addtocontents{toc}{\protect\setcounter{tocdepth}{2}}

\subsection*{Relationship with known work} 
\begin{itemize}
    \item Vaish has shown in \cite{vaish2018motivic} that the motive $EM^{F}_{X}$ satisfies the stronger characterization of motivic intersection complex in the case of Shimura threefolds. Our result generalizes this to arbitrary threefolds - using different methods - thus recovering his result as a special case. Note that \cite{vaish2018motivic} also contains cases not covered by the present article (e.g. Siegel six-folds), and the two articles should thus be treated as independent.
    \item The t-structure constructed in Theorem \ref{B-int} coincides with the one defined by Ayoub and Viale in \cite{ayoub20091motivic}. Specifically, they show that the inclusion $DM^{coh}_{1-mot}(k) \hookrightarrow DM^{coh}(k)$ admits a right adjoint, where $DM^{coh}_{1-mot}(k):=\{p_{*}1_{X}\mid p:X \rightarrow k \;\text{proper}\;, \dim X \leq 1\}$. This leads to a t-structure $t^{1}_{M}$ whose negative part is $DM^{coh}_{1-mot}(k)$. It is easy to see that the t-structure $(^{w}DM^{\leq 1+alg}(k),\;^{w}DM^{>1+alg}(k))$ coincides with $t^{1}_{M}$, but our construction offers a more palatable formulation of Ayoub-Viale’s method. This reformulation also allows us to recover the refined Chow–Künneth decomposition of surfaces introduced in \cite{murre_kahn_pedrini}.
    \item Wildeshaus in \cite{wildeshaus_shimura_2012} has constructed the intersection complex for motives of abelian type, where the resolutions lie in the conservative category $DM^{Ab}(-)$, the primary examples being Shimura varieties. Our work not only explores the different settings and provides an example where the fibers potentially involve phantom motives, i.e., motives outside the conservative category; but also strengthens their results for Shimura threefolds, since we are able to prove the stronger criteria of \cite{wildeshaus_ic} and get a \emph{functorial} intersection complex.
\end{itemize}

\subsection*{Outline} 
The organization of the paper is as follows.\\
Section 2.1 recalls the gluing of t-structures and weight structures on a triangulated category. Section 2.2 reviews the category of Chow (pure) motives and the Chow–Künneth decomposition for varieties. Sections 2.3--2.4 summarize essential facts about the triangulated category of mixed motivic sheaves and their realization functors. Section 2.5 discusses Wildeshaus’s result on the conservativity of realization functors for abelian motives.

Section 3.1 recalls the analogue of Morel’s weight truncation for a threefold $X$ in the motivic setting and establishes several results concerning $EM^{F}_{X}$. Sections 3.2--3.3 present the construction of the refined version of Morel's truncation and examine its realization in the bounded derived category of mixed Hodge structures.

Sections 4.1--4.2 revisit the work of Wildeshaus, giving the construction of the motivic intersection complex for the motive of abelian type (up to dimension 1). This complex is constructed on the complement of a certain finite closed subvariety of a threefold. Section 4.3 shows that it coincides with $EM^{F}_{U}$. Finally, Section 4.4 contains the main result: $EM^{F}_{X}$ is a Chow motive, proved using weight conservativity together with the refined truncations constructed in Section 3.2. 
\subsection*{Notation}Throughout, all schemes $X$ are assumed to be Noetherian, separated, reduced of finite type over a field $k$
of characteristic zero. If a different base field is considered, it will be stated explicitly. Additional notation will be introduced as needed in the relevant sections.

\subsection*{Acknowledgement}The authors thank IISER Mohali for providing a stimulating research environment and for access to excellent departmental library and computational facilities, supported by the DST-FIST grant SR/FST/MS-I/2019/46(C).
\section{Preliminaries}\label{sec:prelim}
\subsection{Preliminaries on \texorpdfstring{$t$}{t}-structures}
\begin{definition}\label{def:mstruct}
A \emph{$t$-structure} on a triangulated category $D$ (see \cite{BBD}) is a pair of full subcategories $(D^{\le }, D^{>})$ satisfying three properties:
\begin{itemize}
\item (Orthogonality) $\hom (a,b)=0, \forall a\in D^{\le }, b\in D^{>}$
\item (Invariance) $D^{\le }[1]\subset D^{\le }$, and $D^{> }[-1]\subset D^{>}$
\item (Decomposition) $\forall a\in D$, there is a distinguished triangle $a_{\le}\rightarrow a\rightarrow a_>\rightarrow $ with $a_{\le}\in D^{\le }$ and $a_>\in D^{>}$. 
\end{itemize}
A \emph{weight structure} (also called a co-$t$-structure) on a triangulated category $D$ (see \cite{bondarko_weights}, for example) is the same as a $t$-structure, except that instead of invariance, it satisfies co-invariance, and we need an additional condition for closure under summands (which is automatic for $t$-structures):
\begin{itemize}
	\item (Karaubi-closed) $D^{\le }$ and $D^{>}$ are closed under taking summands.
	\item (Orthogonality) $\hom (a,b)=0, \forall a\in D^{\le }, b\in D^{>}$
	\item (Co-invariance) $D^{\le }[-1]\subset D^{\le }$, and $D^{> }[1]\subset D^{>}$
	\item (Decomposition) $\forall a\in D$, there is a distinguished triangle $a_{\le}\rightarrow a\rightarrow a_>\rightarrow $ with $a_{\le}\in D^{\le }$ and $a_>\in D^{>}$. 
\end{itemize}
We define an \emph{$m$-structure} to be a pair of full subcategories which form both a $t$-structure and a weight structure. In particular, for an $m$-structure $D^{\le }$ and $D^>$ are triangulated subcategories and $a\mapsto a_{\le}$ as well as $a\mapsto a_>$ are triangulated functors. 
\end{definition}

\begin{notation}
By $\infty$ (resp. $-\infty$) we mean the $t$-structure $(0,D)$ (resp. $(D,0)$). For a $t$-structure named $t$, the truncation functors will be denoted by $\tau_{\le t}$ and  $\tau_{>t}$ respectively. For a $m$-structure, the truncation functors will sometimes also be denoted by $w_{\le t}$ and $w_{>t}$ respectively.
\end{notation}
\begin{definition}
	Let $D$ be a triangulated category, and $S\subset D$ be a collection of objects of $D$. We define $\vvspan{S}$, the span of $S$, to be the smallest triangulated subcategory of $D$ containing $S$ which is closed under taking summands. We do not insist $\vvspan S$ to be closed under arbitrary direct sums. 
	
	The objects of $\vvspan S$ can be constructed by (finitely many iterations of) taking shifts, extensions and summands of objects of $S$. 
\end{definition}
   \begin{lemma} \label{M generated by finite objects}
        Let $D$ be a triangulated category and $S$ be a collection of objects in $D$ with $0 \in S$. For any object $M \in \langle S \rangle $, there exists finite collection $S'$ of objects in $S$ such that $M \in \langle S' \rangle$   \end{lemma}
        \begin{proof}
        Let $D'\subset D$ be the collection of objects $M\in \langle S\rangle$ such that there exists a finite subcollection $S'$, depending on $M$, such that $M\in \langle S' \rangle$. Then it is immediate that $S\subset  D'$ (for any $M\in S$, take $S'=\{M\}$), is closed under shifts ($S'$ for $M[1]$ is same as that for $M$) and taking triangles ($S'$ for $\cone(M\rightarrow M')$ is union of that for $M$ and $M'$). Thus, $D'$ is triangulated and by definition, $D\subset D'$. 
        \end{proof}
        
We also have the following proposition which can be proved by an easy induction.
\begin{proposition}[$m$-structures from generators]\label{tFromGenerators}
	Let $A,B, H\subset D$ be a collection of objects of a triangulated category $D$. Assume $\hom(A,B[n])=0$ for all $n\in \Z$, and
	\[
		h\in H\Rightarrow \exists a\in A, b\in B\text{ such that there is}\text{ a distinguished triangle }a\rightarrow h\rightarrow b\rightarrow.
	\]
	Then if  $\vvspan A, \vvspan B\subset \vvspan H$, the pair $(\vvspan A , \vvspan B )$ is a $m$-structure on the triangulated subcategory $\vvspan H $.
\end{proposition}
\begin{proof}
	See \cite[2.1.4]{vaish2016motivic}. 
\end{proof}
We will also have the occasion to use the following:
\begin{lemma}\label{Pre on t-structures}
  Let $D$ be a triangulated category with two t-structures $({D}^{\leq t_{1}}, {D}^{>t_{1}})$ and $({D}^{\leq t_{2}}, {D}^{>t_{2}})$ such that ${D}^{\leq t_{1}} \subset {D}^{\leq t_{2}}$ (equivalently ${D}^{>t_{2}} \subset {D}^{>t_{1}}$). Then for any $M \in D$, \begin{align*}
  \tau_{\leq t_{1}}\tau_{\leq t_{2}}M \cong \tau_{\leq t_{1}}M  
  \end{align*}
\end{lemma}
\begin{proof}
   1.3.4 of \cite{BBD}  
\end{proof}
\begin{para}\label{gluing:situation}
	We fix a scheme $X$ (which will quickly be assumed to be Noetherian of finite type over a field $k$) and work with (versions of) the derived category of (motivic) sheaves on sub-schemes $W$ of $X$. More generally, we assume that for each $W\hookrightarrow X$ we are given a triangulated subcategory $D_W$, and for each $f:W\hookrightarrow W'$, adjoint pairs:
	\begin{align*}
		f^*:D_{W'}\leftrightarrows D_W:f_* & & f_!:D_W\leftrightarrows D_{W'}:f^!
	\end{align*}
	which satisfy the formalism of Grothendieck's four functors. We will also assume that for each (Zariski) point $\epsilon:\Spec L\hookrightarrow X$, we have a triangulated category $D(L)$ such that continuity holds (see \cite[\S 3]{vaish2017punctual} for a detailed summary). In particular there is a pullback $\epsilon^* : D_X\rightarrow D(L)$ and in fact $D(L)$ is the two-limit $\underset{U\supset \Spec L} {2\lim}D(U)$ where $U\subset X$ are open. 
	
	If $Y$ denotes the closure of $\epsilon$ in $X$, we define:
	\begin{align*}
		\epsilon^! := \epsilon_Y^*f^!:D_X\rightarrow D(L) & &\text{ where }& &\epsilon:\Spec L \overset{\epsilon_Y}\hookrightarrow Y\overset f\hookrightarrow X\text{ is the natural factorization.}
	\end{align*}
\end{para}
\begin{definition}[Gluing]\label{ordinaryGluing}
 Assume the situation of \ref{gluing:situation} on a scheme $X$.  Let $\mathcal S=(S_{0}, S_{1}, \dots, S_{r})$ be a stratification of $X$, that is $X = S_{0}\sqcup \cdots \sqcup S_{r}$, $S_{0}$ is open in $X$; $S_{1}$ is open in complement of $S_{0}$; $S_{2}$ is open in complement of $S_{0}\sqcup S_{1}$ and so on and so forth. Then due to \cite[Theorem 1.4.10]{BBD} there, given a $t$-structure $t_{i}$ (resp. weight structure, resp. $m$-structure) on $D(S_{i})$, we get a glued $t$-structure (resp. weight structure, resp. $m$-structure) on $D_{X}$ which we denote as $t:=(t_{1},\dots, t_{r})$ in the sequel, and is given by the following: 
\begin{align*}
	D^{\le (t_{1}, \dots,t_{n})}(X) &:= \{a\in D_{X}\mid f_{i}^{*}(a)\in D^{\le t_{i}}(S_{i})\text{ for each }f_{i}:S_{i}\hookrightarrow X\}\\
	D^{> (t_{1}, \dots,t_{n})}(X) &:= \{a\in D_{X}\mid f_{i}^{!}(a)\in D^{> t_{i}}(S_{i})\text{ for each }f_{i}:S_{i}\hookrightarrow X\}
\end{align*}
\end{definition}
\begin{notation}
    We refer to the truncation functors associated to the glued t-structure on $D_{X}$ as $\tau_{\leq (t_{1}, \dots, t_{m})}$ and $\tau_{> (t_{1}, \dots, t_{m})}$.
\end{notation}
We record the following for later use.
\begin{lemma}\label{j* commutes with truncation functor}
 Let $X$ be a scheme equipped with the formalism of Grothendieck's four functors, together with a stratification given by an open immersion $j: U \hookrightarrow X$ and its closed complement $i: Z \hookrightarrow X$. Suppose that $D_{U}$ and $D_{Z}$ are endowed with t-structures $t_{1}$ and $t_{2}$ respectively. Then, for any $M \in D_{X}$, we have the following isomorphisms:
 \begin{enumerate}[(i)]
     \item $j^{*} \tau_{\leq (t_{1},t_{2})}M \cong \tau_{\leq t_{1}}\,j^{*}M$ and $j^{*} \tau_{> (t_{1},t_{2})} M \cong \tau_{>t_{1}}\, j^{*}M$.
     \item If $t_{1}= \infty$, then $ i^{*} \tau_{\leq (\infty,t_{2})} M \cong \tau_{\leq t_{2}}\; i^{*} M$
 \end{enumerate}
 \end{lemma}
 \begin{proof}
 Consider the following diagram (see \cite{BBD} 1.4.10) yielding the decomposition triangle of the glued t-structure :
     \begin{figure}[H]\label{gluing diagram}
\centering
\begin{tikzpicture}[scale=0.8, transform shape]
  \node (Kp)  at (0,0)   {$\tau_{\leq (t_{1},t_{2})}M$};
  \node (Y)   at (2,2)   {$Y$};
  \node (T)   at (4,4)   {$i_{*}\tau_{>t_{2}} \,i^* Y$};
  \node (Kpp) at (6,2)   {$\tau_{> (t_{1},t_{2})}M$};
  \node (J)   at (8,0)   {$j_* \tau_{>t_{1}}\, j^{*}M$};
  \node (K)   at (3.9,1.2)   {$M$};

  \draw[->] (Kp) -- (Y);
  \draw[->] (Y) -- node[pos=.55,left] {} (T);
  \draw[->,dashed] (T) -- (Kpp);
  \draw[->,dashed] (Kpp) -- (J);

  \draw[->] (Y) -- (K);
  \draw[->] (Kp) -- node[pos=.55,above] {} (K);
  \draw[->] (K) -- (Kpp);
  \draw[->] (K) -- node[pos=.5,below] {} (J);
\end{tikzpicture}
\end{figure}
Applying the functor $j^{*}$ to the diagram, and using  $j^{*}j_{*}=id$, we obtain the distinguished triangle: 
\[j^{*}Y \rightarrow j^{*}M \rightarrow \tau_{> t_{1}}\,j^{*}M \rightarrow\]
By the uniqueness of the cone, it follows that $j^{*}Y \cong \tau_{\leq t_{1}} \,j^{*}M$.
Moreover, since $j^{*}i_{*}=0$, the top triangle in the diagram yields $j^{*}Y \cong j^{*} \tau_{\leq (t_{1},t_{2})}M$.
Combining the two, we obtain the desired isomorphism.

The other case proceeds similarly.

When $t_{1}= \infty$, we have a distinguished triangle
    \begin{align}\label{gluing triangle}
    \tau_{\leq (\infty,t_{2})} M \rightarrow M \rightarrow  i_{*} \tau_{>t_{2}}\; i^{*} M \rightarrow
 \end{align}
 Applying the functor $i^{*}$ to this triangle and using isomorphism $i^{*}i_{*} \cong id$ yields,
  \[ i^{*} \tau_{\leq (\infty,t)} M \cong \tau_{\leq t}\; i^{*} M. \]
  This completes the proof.
 \end{proof}
\begin{definition}\label{gluing:spreadingOut}
		Assume the situation of \ref{gluing:situation} on a scheme $X$. Further assume that for each Zariski point $\epsilon: \Spec K\hookrightarrow X$, we are given a $t$-structure (resp.weight structure, resp. $m$-structure) $(D^{\le}(K), D^{>}(K))$ on a full subcategory $D'(k)\subset D(k)$. 
		For any $U\hookrightarrow X$, define
		\begin{align*}
			D^{\le}(U)	:=\{a&\in D_U\big| \epsilon^*(a)\in D^{\le}(K)\text{ for }\epsilon:\Spec K\rightarrow U\text{ any point of }U\} \\
			D^{>}(U)	:=\{a&\in D_U\big| \epsilon^!(a)\in D^{>}(K)\text{ for }\epsilon:\Spec K\rightarrow U\text{ any point of }U\}\\
			D'(U) := \{a&\in D_U\big| \exists b\rightarrow a\rightarrow c\rightarrow \text{ s.t. }b\in D^{\le}(U), c\in D^{>}(U)\}
		\end{align*}
		as full subcategories. In particular if $f:S\hookrightarrow T$ is an immersion, $S, T\in Sub(X)$:
		\begin{align*}
			f^*(D^{\le}(T))\subset D^{\le}(S)& &f^!(D^{>}(T))\subset D^{>}(S)
		\end{align*}
\end{definition}
	The subcategories $D^{\le}(U)$ (resp. $D^{>}(U)$) are said to satisfy continuity if for any point $\tilde \epsilon: \Spec K \rightarrow X$ and any $a\in D^{\le}(K)$ (resp. $D^{>}(K)$) there is a neighborhood $U$ of $\epsilon$, that is an open set $U$ such that $\tilde \epsilon$ factors through $\epsilon: \Spec K\rightarrow U$, and an object  $\bar a \in D^{\le U}$ (resp. in $D^{>}(U)$) such that $\epsilon^{\ast}(\bar a) = a$. 

	We then have the following more useful, but minor, refinement of punctual gluing \cite{vaish2017punctual}:
	\begin{proposition}[Proposition 2.1.6 of \cite{vaish2018motivic}]\label{gluing:mainresult}
		Assume the situation of \ref{gluing:spreadingOut} on a Noetherian scheme $X$, and in particular for each point $\Spec k\rightarrow X$, we are given a $t$-structure (resp. weight structure, resp. $m$-structure) $(D^{\le}(K), D^{>}(K))$ on $D(k)$. We also assume that the systems $D^{\le}(U)$ and $D^{>}(U)$ as above, satisfy continuity. Then $D'(X)$ is a pseudo-abelian triangulated subcategory of $D(X)$ and 
		the pair $(D^{\le}(X), D^{>}(X))$ forms a $t$-structure (resp.weight structure, resp. $m$-structure) on $D'(X)$.
	\end{proposition}		
    \begin{remark}
	In using the previous proposition, the hard task would be to determine $D'(X)$ explicitly. It does not seem likely that the categories $D'(U)$ would automatically satisfy the formalism of four functors. We will be completely bypassing the question of determining $D'(X)$ in our cases of interest and instead work with objects (and subcategories) which will be known to be in $D'(X)$.
\end{remark}
\subsection{Chow motives}\label{sec:chow}
	We refer the reader to Scholl's exposition  \cite{scholl1994classical} for an exposition on the classical category of Chow motives. 
	
	In particular given any field $k$, we have a category of Chow motives denoted $CHM(k)$ (resp. effective Chow motives denoted $CHM^{eff}(k)$), whose objects are given by triples $(X,p,n)$ (resp. tuples $(X,p)$) with $X$ a smooth, projective over $k$, $p$ a correspondence such that $p\circ p = p$ and $n\in \Z$. Here $\circ$ denotes the composition of correspondences. Morphisms in this category are given by appropriate correspondences, composing under $\circ$ (see \cite[1.4]{scholl1994classical} for details).
$CHM^{eff}(k)$ resp. $CHM(k)$ are additive, pseudo-abelian categories. There are natural functors:
		\[
			(SmProj/k)^{op} \xrightarrow h CHM^{eff}(k)\hookrightarrow CHM(k)
		\]
		where $h(X)=(X,\Delta_X)$ ($\Delta_X$ being the class of diagonal) while $h(f):=f^*$ is the transpose of the graph of $f:Y\rightarrow X$ in $X\times Y$. Also $CHM^{eff}(k)\hookrightarrow CHM(k)$ is given by $(X,p)\mapsto (X,p,0)$.
		\begin{para}\label{chow:duality}
	 $CHM(k)$ is a rigid tensor additive category ({see \cite[1.15]{scholl1994classical}}) under tensor $\otimes$ given by 
	 	\[
			(X,p,m)\otimes (Y,q,n) = (X\times Y,p\times q,m+n)
		\]
		and dual of $(X,p,m)$ for $X$ irreducible of dimension $d$, is:
	 	\[
			(X,p,m)^{\vee} = (X,{^{t}p}, d- m)
		\]
	 In particular internal $\hom$, denoted $\underline\hom$, is right adjoint to $\otimes$  and we have,
	 	\[	\hom(M\otimes N, P) = \hom (M, \underline \Hom(N, P)) = \hom (M, N^\vee\otimes P) \]
	The Lefschetz motive $\mathbb L$ is defined as the object $(\Spec k, \Delta_{\Spec k}, -1)$ in $CHM(k)$. It lives in $CHM^{eff}(k)$ because of the decomposition $h(\mathbb P^1) \cong h(\Spec k)\oplus \mathbb L$ ({see \cite[1.13]{scholl1994classical}}). This object becomes invertible in $CHM(k)$ for the tensor product.
\end{para}
\begin{npara}[Chow-K\"unneth decomposition]\label{piconstruction} 
	Let $X$ be pure of dimension $d$. In \cite{scholl1994classical} Scholl constructs projectors $\pi_i(X)$ for $i\in \{0,1,2d-1,2d\}$.
	If $\dim X\le 2$, we also have $\pi_2(X)$ (these projectors are originally due to Murre In \cite{murre1990motive, murre1993conjectural}). Define:
		\[
			h^i(X):=(X, \pi_i(X)) \in CHM^{eff}(k)
		\]
		for $i\in{0,1,2d-1,2d}$ and any  $X$ or $i=2$ and $\dim X\le 2$. Then, under suitable realization functors $h^{i}(X)$ realize to $i$-th cohomology of $X$. 
		More generally it is expected that there are summands $h^{i}(X)$ of $h{(X)}$ for $0\le i\le 2d$ for $d=\dim X$ such that
		\[
			h(X) := (X,\Delta_{X}, 0) \cong \bigoplus_{i=0}^{2d}h^{i}(X)
		\] 
		where $\Delta_X$, the diagonal inside $X\times X$, is the identity projector on $X$, and the construction is expected to be such that $h^{i}(X)$ realize to the $i$-th cohomology of $X$ under any suitable realization functor. This is the conjectural Chow–Künneth decomposition of $X$. By the above, it is known for surfaces and is the key ingredient of the construction of the motivic intersection complex in \cite{vaish2017weight}. In cases where it is known, we also define the following useful notation: 
		\begin{align*}
			h^{{\le a}}(X(-r)) := \bigoplus _{i=0}^{a-2r}h^{i}(X)(-r)& &h^{{> a}}(X(-r)) := \bigoplus _{i=a+1-2r}^{2\dim X}h^{i}(X)(-r)
		\end{align*}
		where by $(-r)$ we mean the Tate twist operation taking $(X,p,s)$ to $(X,p,s-r)$. 
\end{npara}
\begin{remark}While it is not made explicit in the original articles of Murre (or Scholl), it follows from the constructions that for $i=0, 1$ (resp. $i=2$) $h^{ i}$ is a functor from the category of smooth projective varieties (resp. smooth projective varieties of $\dim \le 2$) to the category of (effective) Chow motives, see Remark \cite[2.4.7]{vaish2017weight}.
\end{remark}
\subsection{Motivic Sheaves}\label{sec:newmotives}
\begin{para} \label{intro:dmx}
Given any base scheme $S$, there exists a rigid tensor triangulated category of motivic sheaves $DM(S)$ with unit object denoted $1_S$, Tate twists denoted $A\mapsto A(r)$ and such that the formalism of Grothendieck's six functors holds. 

One choice for such a construction is the category of motivic sheaves without transfers as constructed by Ayoub in \cite{ayoub_thesis_1, ayoub_thesis_2}. This is the category $\mathbb{SH}_\mathfrak{M}^T(S)$ of \cite[4.5.21]{ayoub_thesis_2} with $\mathfrak{M}$ being the complex of $\Q$-vector spaces (and one works with the topology \'etale topology), also denoted as $DA(S)$ in the discussion \cite[2.1]{ayoub2012relative}. To play well with the realization functors and continuity, we will instead restrict attention to the subcategory of compact objects in $DA(S)$, which are also stable under Grothendieck's four functors. The second choice for the construction is the compact objects in the category of motivic sheaves with transfers, the Beilinson motives $DM_{B,c}(S)$ as described in the article \cite{Cisinski_2019}. Again, the objects in $DM_{B,c}(S)$ are compact by construction and play well with realizations.

We refer the reader to \cite[\S 2.6]{vaish2017weight} for a summary of properties of the category (except about realizations) which will be used here.
\end{para}
\begin{para}
Over a field $S=\Spec k$, the construction of $DM(\Spec k)$ is due to Voevodsky \cite{voevodsky}, who also shows that there are natural functors:
	\begin{equation*}
			(SmProj/k)^{op}\overset {h_k} \longrightarrow CHM(k)\hookrightarrow	DM(\Spec k)
	\end{equation*}
	where $SmProj$ denotes the category of smooth projective varieties over $\Spec k$, $CHM(k)$ denotes the category of Chow motives over the field $k$, and the last functor is fully faithful. This can be extended to give a functor:
	\[
		h_k: (Sch/k)^{op} \rightarrow DM(\Spec k)
	\]
	\para{} More generally, due to work of Corti-Hanamura \cite{cortiHanamura}, there is a category of Chow motives $CHM(S)$ over any regular base $S$. Due to work of Hebert \cite{hebert2011structure} and Bondarko \cite{bondarko2014weights} there is a weight structure $(DM^{w\le 0}(S), DM^
    {w>0}(S))$ on $DM(S)$ and due to the work of Fangzhou \cite{fangzhou2016borel}, it's heart can be identified with the category of Chow motives. Therefore, we have functors
	\begin{equation*}
		(SmProp/S)^{op}\overset {h_S} \longrightarrow CHM(S)	\hookrightarrow DM(S) \text{ with essential image as heart }DM_{w\le 0}(S)\cap DM_{w\ge 0}(S)
	\end{equation*}
	where $SmProj/S$ is the category of proper schemes over $S$ which are smooth, and the last functor is fully faithful. This can be extended to give a functor:
	\[
		h_S: (Sch/S)^{op} \rightarrow DM(S)	\text{ with }(\pi:X\rightarrow S)\mapsto \pi_*1_X
	\]
\end{para}
\begin{para}
	As a matter of notation, we will often identify the objects in $CHM(S)$ and $DM(S)$ via the above fully faithful embedding. We will often write $h$ for $h_S$ when the base scheme is clear from the context.
\end{para}
We shall need the following fact about the Chow weight structure, which we state without proof.
\begin{proposition}[Theorem 4.3.2 II, Lemma 1.3.8 \cite{bondarko_weights}]\label{induced chow weight struct}Let $H$ be the full additive subcategory generated by some family of objects in $CHM(X)$. Let $C$ be the triangulated category generated by $H$ in $DM(X)$. Then the Chow weight structure $(DM^{w\leq 0}(X),DM^{w>0}(X))$ on $DM(X)$ restricts to $C$ with heart as Karaoubi closure of $H$.
\end{proposition}
\subsection{Realizations}\label{realization} We recall the various realization functors of the triangulated category of motives (with rational coefficients) which will be relevant in our context. They provide a bridge connecting the abstract theory of motives with classical (co)homology theories.

We begin with the covariant (homological) realization functors. For a proper morphism $f:Y \rightarrow Y'$  of $X$- schemes, such a realization functor (if exists) preserves the counit map $p_{*}f_{!}f^{!}1_{Y'} \rightarrow p_{*}1_{Y'}$ of adjunction $(f_{!},f^{!})$, where $p$ is the structure morphism. We have the following examples:

Let $X$ be a scheme over a field $k$ of characteristic zero together with a fixed embedding $\sigma:k \hookrightarrow \mathbb{C}$. 
\begin{itemize}
     \item There exists a Betti realization functor \cite{Ayoub2010} (see \cite[Ex~ 17.1.7]{Cisinski_2019} for the bounded target)
     \[R_{B}: DM(X) \rightarrow D^{b}_{c}(X,\mathbb{Q})\] where $D^{b}_{c}(X,\mathbb{Q})$ is the bounded derived category of $\mathbb{Q}$-constructible sheaves of on $X^{an}$. When $X=\Spec k$, it sends the motive associated to a smooth projective variety $Y$ to the singular chain complex of $Y^{an}$ with rational coefficients. These categories carry the formalism of Grothendieck’s six functors, and the functor $R_{B}$ is 
    compatible with these operations. 
          \item  There is also a Hodge realization functor (see \cite[\S 1.3, Fact G]{MR3751289}) to the bounded derived category of mixed Hodge structures:
   \[R_{\mathcal{M}}: DM(k) \rightarrow D^{b}_{\mathbb{Q}}(MHS(k)).\]
This functor assigns a motive of a smooth projective variety $p:X \rightarrow k$ to a complex $R_{\mathcal{M}}(p_{*}1_{X})$ of mixed Hodge structures, which computes the singular homology of $X^{an}$ endowed with their natural mixed Hodge structures.

A relative version of this theory has been developed by Saito \cite{MR1047415} in the form of mixed Hodge modules over a base scheme $X$. The bounded derived category $D^{b}(MHS(X))$ of mixed Hodge modules admits the formalism of six functors, analogous to the case of constructible sheaves. There is a natural functor
\[ \For: D^{b}(MHS(X)) \rightarrow D^{b}_{c}(X, \mathbb{Q})\] which forgets the Hodge-theoretic structure. This functor is both faithful and conservative. However, note that the realization functor from
$DM(X)$ to the category $D^{b}(MHS(X))$ has not yet been constructed in general.
\item  The $\ell$-adic realization \cite[\S 7.2]{cisinski2016etale} is the functor
\[ R_{\ell}:DM(X) \rightarrow D^{b}_{c}(X, \mathbb{Q}_{\ell})\] 
where the target is the bounded derived category 
of constructible $\mathbb{Q}_{\ell}$ sheaves on $X$ \cite{ekedahlAdic}. The functors $R_{\ell}$ are symmetric monoidal, and they commute with the functors $f_{*}, f^{*},f_{!},f^{!}$ \cite[Theorem 7.2.24]{cisinski2016etale}.
\end{itemize}
All of the above realization categories admit a perverse t-structure with heart denoted as $Perv(X)$. They give rise to cohomological functors
\[ ^{p}H^{i}\circ R_{*}: DM(X) \rightarrow Perv(X).\]
where $R_{*} \in \{R_{B}, R_{\mathcal{M}}, R_{\ell}\}$. The category $D^{b}_{c}(X, \mathbb{Q}_{l})$ does not carry any natural weight structure. One can resolve this partially using a result of Bondarko \cite[2.1.2]{bondarko_weights}, which allows one to define weights for complexes in the image of the realization functor. Specifically, for any $M \in DM(X)$, there is a canonical weight filtration of $^{p}H^{i}R_{\ell}(X)$ induced by any choice of weight decomposition of $M[j]$ in $DM(X)$. Moreover, for any morphism $f: M \rightarrow M'$ in $DM(X)$, the morphism $^{p}H^{i}R_{\ell}(f)$ is strict with respect to these canonical filtrations. 

\textbf{Important}: Note however that, it is entirely possible that a complex below may be realized (up to isomorphism) by two (non isomorphic) motives and hence may carry two different weight structures. This however, would not cause a problem for our arguments.\\

We now turn to the contravariant/cohomological realizations: For a proper morphism $f:Y \rightarrow Y'$  of $X$- schemes, such a realization functor preserves the unit map $p_{*}1_{Y'} \rightarrow p_{*}f_{*}f^{*}1_{Y'}$ of adjunction $(f^{*},f_{*
})$, where $p$ is the structure morphism.
\begin{itemize}
    \item Huber in \cite[Theorem 2.3.3]{MR1775312} constructed the triangulated category of mixed realizations $ D_{\mathcal{MR}}$ and the triangulated contravariant functor \[
R_{\mathcal{MR}}: DM(k) \rightarrow D_{\mathcal{MR}}.\]
The category $D_{\mathcal{MR}}$ unifies the singular, étale and de Rham cohomology, that is, all these cohomology theories factor through $D_{\mathcal{MR}}$. In particular, there are contravariant realization functors (see \cite[Corollaries 2.3.4, 2.3.5]{MR1775312}):
\begin{align*}
    DM(k) \rightarrow D^{b}_{c}(X, \mathbb{Q}_{\ell}) \;\;\text{and}\;\;DM(k) \rightarrow D^{b}(MHS(k)).
\end{align*}
Ivorra \cite{ivorra2007realisation} later constructed an integral $\ell$-adic realization functor and showed that, over a field of characteristic zero, it agrees with Huber's $\ell$-adic realization for rational mixed motives.
\end{itemize}
This section sets the notation for the realization functors and will be referred to when needed; these functors play a central role in Proposition \ref{realizations of EM at each point}.
\subsection{Abelian motives}
In this section, we discuss some properties satisfied by Abelian motives. We begin by recalling the notion of finite dimensional Chow motives.
\begin{definition}
    Let $M \in CHM(X)$. The motive $M$ is evenly finite dimensional if there exists a positive integer $n$ such that $\bigwedge^{n}M = 0$. The motive $M$ is oddly finite dimensional if there exists a positive integer $n$ such that $Sym^{n}M = 0$.

The motive $M$ is finite dimensional if $M$ can be written as a direct sum $M =M^{+} \oplus M^{-}$ where $M^{+}$ is evenly finite dimensional and $M^{-}$ is oddly finite dimensional.
\end{definition}
\begin{example}\label{kimura finiteness of h(C)} Motive of a curve. More generally, the motive of any abelian variety over $k$ is finite dimensional (\cite[Corollary 4.4, Example 9.1]{MR2107443}).
\end{example}
The following lemma shows the behaviour of Kimura finite objects under pullbacks and pushforwards. Consider the full dense additive subcategory $CHM_{s}(Y)$ of $CHM(Y)$ generated by the collection of objects $\{h_{Y}(Y') \mid Y' \rightarrow Y \;\text{smooth proper and} \;Y'\; \text{regular}\}$. 
\begin{lemma}\label{stability of kimura objects} Let $f: X \rightarrow Y$ be a morphism of regular schemes over $k$. Let $M \in CHM_{s}(Y)$. Then the following holds:
    \begin{enumerate}[(i)]
    \item If $f$ is dominant and $f^{*}M$ is finite dimensional then $M$ is also finite dimensional.
    \item If $f$ is finite and etale then $f_{*}M$ is finite dimensional.
    \end{enumerate}
\end{lemma}
\begin{proof}
Refer to \cite[Page 50]{osullivan2009algebraiccyclesabelianvariety}
\end{proof}
\begin{corollary}\label{motive of abelian scheme is f.d}
    Let $p: Y \rightarrow X$ be a proper smooth morphism with $Y$ regular. Assume that $X$ is irreducible and the dimension of the generic fibre is $\leq 1$. Then $p_{*}1_{Y}$ is finite dimensional.

    More generally, if $Y$ is an abelian scheme over an irreducible regular scheme $X$. Then, $p_{*}1_{Y}$ is finite dimensional.  
\end{corollary}
\begin{proof}
    Let $\eta: \Spec K \hookrightarrow X$ be the generic point of $X$. By proper base change, we have $\eta^{*}p_{*}1_{Y} \cong p_{*}1_{Y_{\eta}}$, where $Y_{\eta}$ is fibre over $\eta$. Since the generic fibre $Y_{\eta}$ has dimension $\leq 1$ (or is an abelian variety in the case of an abelian scheme), its motive $p_{*}1_{Y_{\eta}}$ is finite dimensional. Then, from Lemma \ref{stability of kimura objects} $(i)$ it follows that $p_{*}1_{Y}$ is also finite dimensional.
\end{proof}
Finally, we recall the following conservativity result for the realization functors, due to Wildeshaus.
 \begin{proposition}
Let $X$ be a regular scheme of finite type over $k$. Then the restriction
of the ($\ell$-adic or Betti) realization functor to the category
    \[DM^{Ab}(X) := \langle h_{X}(Y)(j) \in DM(X) \mid Y/X\; \text{is an abelian scheme},j \in \mathbb{Z}\rangle\]
is conservative. Moreover the restriction of $\ell$-adic realization functor $R_{\ell}$ to $DM^{Ab}(X)$ is weight conservative. More precisely, for $M \in DM^{Ab}(X)$,
    \begin{itemize}
        \item $M$ lies in $DM^{w \leq a}(X)$ if and only if $^{p}H^{i}(R_{\ell}M)$ has weight $\leq i+a$ for all $i \in \mathbb{Z}$.
        \item $M$ lies in $DM^{w>a}(X)$ if and only if $^{p}H^{i}(R_{\ell}M)$ has weight $> i+a$ for all $i \in \mathbb{Z}$.
    \end{itemize}
\end{proposition}
\begin{proof}
From definition 3.5(a) of \cite{wildeshaus_conservativity}, the motive $h_{X}(Y)(j)$ is of abelian type. Then Theorem 4.3 of loc. cit. implies that the restriction of the realization functor to $DM^{Ab}(X)$ is conservative, while Theorem 4.4 of the same work establishes weight conservativity.
\end{proof} 
\begin{remark}
From the above proposition, it follows 
that Hodge realization $R_{\mathcal{M}}|_{DM^{Ab}(k)}$ is also conservative: if $M \in DM^{Ab}(k)$ satisfies $R_{\mathcal{M}}(M)=0$, then $R_{B}(M)= \For(R_{\mathcal{M}}(M))=0$, and the conservativity of Betti realization then implies $M=0$. 

Moreover, $R_{\mathcal{M}}|_{DM^{Ab}(k)}$ is weight conservative as well (see \cite[Remark 4.6]{wildeshaus_conservativity}).
\end{remark} 
In our situation, weight conservativity will follow as a simple consequence of the conservativity result. We will discuss this version in Proposition \ref{weightconservativity_our}. We will also require the following lemma, which shows that a conservative t-exact functor is t-conservative.
\begin{lemma}\label{conservativity implies t-conservativity}
    Let $F: \mathcal{C} \rightarrow \mathcal{D}$ be a t-exact functor of triangulated categories with t-structures $(\mathcal{C}^{\leq t},\mathcal{C}^{>t})$ and $(\mathcal{D}^{\leq t},\mathcal{D}^{>t})$ respectively. Assume that $F$ is conservative. Then, $F$ is t-conservative, that is, if $F(A) \in \mathcal{D}^{\leq t}$ (resp. $F(A) \in \mathcal{D}^{> t}$),  then $A \in \mathcal{C}^{\leq t}$ ( resp. $A \in \mathcal{C}^{> t}$) for all $A \in \mathcal{C}$.
\end{lemma}
\begin{proof}
    Applying the triangulated functor $F$ to the decomposition triangle $\tau_{\leq t}\,A \rightarrow A \rightarrow \tau_{>t}\, A \rightarrow$ of $A$ yields:
    \[F(\tau_{\leq t}\,A) \rightarrow F(A) \rightarrow F(\tau_{>t}\, A) \rightarrow \]
    By t-exactness of $F$, we have that $F(\tau_{\leq t}\,A) \in \mathcal{D}^{\leq t}$ and $F(\tau_{>t}\, A)\in \mathcal{D}^{>t}$.
    Then the uniqueness of the decomposition triangle and the assumption $F(A)\in \mathcal{D}^{\leq t}$, implies that $F(\tau_{>t}\, A) \cong 0$.
From conservativity of $F$, we then conclude that $\tau_{\leq t}A=0$, which gives the required result.

The other case is identical.
\end{proof}
\section{Refined Morel's weight truncations}
\subsection{Morel's weight truncations} \label{candidate of MIC}
\para In her thesis \cite{morelThesis} S.~Morel observed that the following forms a $m$-structure (and hence a $t$-structure) on the category of mixed sheaves $D^{b}_{m}(X)$:
\begin{align*}
	{^{w}D^{\le i}} := \{ A \mid \forall j \;{^{p}H^{j}(A)} \text{ is of weight}\le i\}& &{^{w}D^{> i}} := \{ A \mid \forall j \;{^{p}H^{j}(A)} \text{ is of weight}> i\}
\end{align*}
	where $^{p}H^{j}$ denotes the $j$-th perverse cohomology. 

Given a stratification $\mathcal S = (S_{1}, \dots, S_{r})$, and a choice of weights $d_{1}, d_{2}, \dots, d_{r}$ we can glue the $t$-structures $({^{w}D^{\le d_{i}}}, {^{w}D^{> d_{i}}})$ on $S_{i}$ to get a glued $t$ structure. We can let $d_{i} = D(\dim S_{i})$ for some function $D$ on non-negative integers (called \emph{weight profile}). 

Then for certain choices of the function $D$ (monotone step functions), the main result of \cite{vvThesis} shows that this $t$-structure is independent of the choice of a (sufficiently fine) stratification, and thus can be denoted as $({^{w}D^{\le D}}, {^{w}D^{> D}})$  on $X$ without ambiguity. Note that if $D(i)=d$ for all $i$, then results of \cite{morelThesis} show that this coincides with the $t$ structure   $({^{w}D^{\le d}}, {^{w}D^{> d}})$ and thus there is no ambiguity in confusing integers and constant functions here.

\begin{para} If $X$ is irreducible with $\dim X=d$ and $j:U \hookrightarrow X$ is smooth open, one of the key results of \cite{morelThesis}, with a refinement due to \cite{vvThesis} is the following: 
\[
	w_{\le d}Rj_{*}1_{U} \cong  IC_{U} \cong w_{\le D}Rj_{*}1_{U}\text{ where }d - 1\le D \le d\text{ with }D(d)=d. 
\]
where $IC_{X}$ denotes the intersection complex (up to a shift). In other words, we can use such a $D$ to recover the intersection complex. The importance of this relation is that for some such weight profiles $D$, it is possible to lift the construction to motives. 

In \cite{vaish2017weight}, the author has constructed the analogue of certain Morel's $t$-structure on specific subcategories of $DM(X)$ using the technique of punctual gluing of $t$-structures, and used that to construct a candidate for the intersection complex. We review this construction below. 
\end{para}
We start by recalling subcategories of $DM(X)$ which are of our interest:
    \begin{definition}
        Let $X$ be a noetherian scheme of finite type over $k$. Define the following collection of objects in $DM(X)$ and their corresponding triangulated categories:
     \begin{align*}
          S^{sm,coh}(X) := &\big\{ p_{*} 1_{Y} \big|\; p: Y \rightarrow X \; \text{projective, $Y$ regular, smooth over}\; k\big\}\\
    S_{d}^{sm,coh}(X) := &\big\{p_{*} 1_{Y}(-r)\;\big|\; p: Y \rightarrow X \; \text{projective, $Y$ connected, smooth over}\; k, \dim Y \leq d, r \geq 0 \big\}\hspace{-15em}\\
           S_{d}^{dom,coh}(X) := &\big\{p_{*} 1_{Y}(-r)\;\big|\; p: Y \rightarrow X \; \text{projective, $Y$ connected, smooth over }k, \dim Y \leq d,\; p \;\text{dominant}, r \geq 0 \big\}\\
DM^{coh}(X) =& \langle S^{sm,coh}(X) \rangle \hspace{1cm} DM^{coh}_{d}(X) = \langle S^{sm,coh}_{d}(X) \rangle \hspace{1cm} 
 DM^{coh}_{d,dom}(X) := \langle S^{dom,coh}_{d}(X) \cup S^{sm,coh}_{d-1}(X) \rangle 
        \end{align*}
    \end{definition}
\begin{para} In \cite{vaish2017weight} the lift of the Morel's t-structure corresponding to a $D$ is first constructed over any point $\Spec K$. Recall \ref{piconstruction} that for any $X$ we can define the Chow–Künneth summands $h^{\le 0}(X), h^{\le 1}(X)$ and $h^{>0}(X), h^{>1}(X)$ of the Chow motive $h(X)$ of $X$. Further, if $\dim X\le 2$, we can even define the summands $h^{\le 2}(X)$ and $h^{>2}(X)$.

Using the fully faithful embedding of Chow motives in $DM(k)$ and the fact that $h(X)$ generate the subcategory of cohomological motives $DM^{coh}(k)$, one can lift certain Morel's $t$-structures:
\begin{align*}
\text{For }i\in \{0,1\}, {^{w}}DM^{\leq i}(k)&:=\langle h^{\leq i}(X)\mid X\in SmProj/k\;\text{and}\; X \;\text{is connected}\rangle \\\;^{w}DM^{>i}(k)&:=\langle h^{>i}(X), h(X)(-r)\mid X\in SmProj/k\;\text{and} X \;\text{is connected}, r\ge 1\rangle
\end{align*}
Furthermore, for weight 2, $h^{2}(X)$, and hence also $h^{\le 2}(X), h^{> 2}(X)$ is well defined provided we have $\dim X\le 2$. Thus, restricting to the smaller subcategory $DM^{coh}_{2}(k)$ we define:
\begin{align*}
^{w}DM_{2}^{\leq 2}(k)&:=\langle h^{\leq 2}(X)\mid X\in SmProj/k\text{ and } X \text{ is connected}, \dim X\le 2 \rangle \\
^{w}DM_{2}^{>2}(k)&:=\langle h^{>2}(X), h^{>1}(X)(-r), h(X)(-r-1)\mid X\in SmProj/k\text{ and }X\text{ is connected}, \dim X\le 2, r\ge 1 \rangle
\end{align*}
\end{para}    
    It was shown in \cite{vaish2017weight} that indeed these lifts do form a $t$-structures:
\begin{theorem}\cite[Corollary 3.2.7]{vaish2017weight} For any field $K$, we have the following $t$-structures on certain subcategories of $DM(K)$:
	\begin{itemize} 
		\item $({^{w}DM^{\le 0}}, {^{w}DM^{> 0}})$ forms a $t$-structure on $DM^{coh}(K)$. 
		\item $({^{w}DM^{\le 1}}, {^{w}DM^{> 1}})$ forms a $t$-structure on $DM^{coh}(K)$. 
		\item $({^{w}DM^{\le 2}}, {^{w}DM^{> 2}})$ forms a $t$-structure on $DM_{2}^{coh}(K)$.
	\end{itemize}
\end{theorem}    
 
Using punctual gluing, they can be further used to produce $t$ structures on appropriate categories for certain weight profiles $F$. In particular, we have the following construction: 
    \begin{theorem} {\cite[Proposition 3.3.5]{vaish2017weight}}\label{def of leq F}
        Fix an irreducible scheme $X$ of dimension 3 and finite type over $k$. Consider the weight profile $F=\{3 \mapsto3, 2\mapsto 3, 1 \mapsto 2, 0 \mapsto2\}$. For any subscheme $Y \subset X$, define $D_{Y}:=DM_{2}^{coh}(Y)$ if $\overline{Y} \neq X$ and $D_{Y} := DM^{coh}_{3,dom}(Y)$ for $Y$ open.
        For any $x= \Spec K \in X$, with $t_{K}$ as the transcendence degree of $K/k$, define:
\[
\begin{array}{ll}
D^{F(t_{K})-t_{k}}(K) = \begin{cases}
DM_{0}^{coh}(K^{perf})& t_{K}=3\\
DM^{coh}_{1}(K^{perf}) & t_{K}=2,1\\
DM^{coh}_{2}(K^{perf}) & t_{K}=0
\end{cases}
&
D^{\leq F(t_{K})-t_{K}}(K):=
\begin{cases}
^{w}DM^{\leq 0}(K^{perf}) & t_{K}=3\\
^{w}DM^{\leq 1}(K^{perf}) & t_{K}=2,1\\
^{w}DM^{\leq 2}_{2}(K^{perf}) & t_{K}=0
\end{cases}
\end{array}
\]
\begin{align*}D^{>F(t_{K})-t_{K}}(K):=\{A \in D(K)|\; Hom(B,A)=0\; \forall\; B \in D^{\leq F(t_{K})-t_{K}}  \} \end{align*}
Then this setup satisfies the formalism of gluing, continuity and continuity of t-structures. In particular the following defines a $t$-structure on $DM^{coh}_{3, dom}(X)$: 
    \begin{align*}
        ^{w}DM^{\leq F}(X)&:= \{ M \in DM^{coh}_{3,dom}(X)\mid s^{*}M \in D^{\leq F(t_{K})-t_{K}}(K)\; \forall\; s=\Spec K \in X\}\\
         ^{w}DM^{> F}(X)&:= \{ M \in DM^{coh}_{3,dom}(X) \mid s^{!}M \in D^{> F(t_{K})-t_{K}}(K)\; \forall\; s=\Spec K \in X\}
    \end{align*}
\end{theorem}
\begin{remark}
The reason for considering the definition as above is the following --  Morel's $t$-structure for weight profile $F$, $({^{w}D}^{\le F}, {^{w}D}^{> F})$ on $D^{b}_{m}(X)$, can also be obtained by punctual gluing of $t$-structures $({^{w}D}^{\le F(t_{K})-t_{K}}, {^{w}D}^{> F(t_{K})-t_{K}})$ on each point $x=\Spec K\hookrightarrow X$ with transcendence degree of $K/k$ being denoted as $t_{K}$. Thus, the definition above should be seen as the motivic lift of the $t$ structure $({^{w}D^{\le F}}, {^{w}D^{> F}})$. 
\end{remark}
\begin{notation}
    The truncation functors for this $t$-structure will be denoted by $w_{\leq F}$ and $w_{>F}$.
\end{notation}
\begin{para} The above t-structure on $DM^{coh}_{3,dom}(X)$ allows us to define an object $EM^{F}_{X}:= w_{\leq F}j_{*}1_{U}$ where $j:U \xhookrightarrow{} X$ is a smooth open dense immersion in $X$. It is independent of the open immersion chosen. This is the candidate for the motivic intersection complex. 
\end{para}
Next, we prove some specific results regarding $EM^{F}_{X}$ which we need in the sequel:
\begin{lemma}\label{EM(X) in terms of EM(U)}
      Let $j: U \xhookrightarrow{}X$ be an open dense immersion. Then $EM^{F}_{X} \cong w_{\leq F}j_{*} EM^{F}_{U}$.
\end{lemma}
\begin{proof}
    First, we will show that if $A \in \;^{w}DM^{> F}(U)$ then $j_{*}A\in \;^{w}DM^{>F}(X)$. Let $s:\Spec K \hookrightarrow X$ be a point in $U$. Then $s^{!}j_{*}A \cong s^{!}j^{!}j_{*}A \cong s^{!}A $. Since $A \in\;^{w}DM^{> F}(U)$, we have $s^{!}A \in\; ^{w}DM^{>F(t_{K})-t_{K}}(\Spec K)$. Now, let $s:\Spec K \hookrightarrow X$ be a point in $Z:=X-U$. Let $i :Z \hookrightarrow X$ denote the closed immersion. Then $s^{!}j_{*}A \cong s^{!}i^{!}j_{*}A \cong 0 \in\;^{w}DM^{>F(t_{K})-t_{K}}(\Spec K)$. Therefore, by Theorem \ref{def of leq F} it follows that $j_{*}$ is left exact.\\
     Next, fix a regular open dense immersion $j':V \hookrightarrow U$. Then we have a distinguished triangle
    \begin{align*}
        EM^{F}_{U} \rightarrow j'_{*}1_{V} \rightarrow w_{>F}j'_{*}1_{V} \rightarrow
    \end{align*}
    Apply the triangulated functor $w_{\leq F}j_{*}$ to this triangle to obtain
    \begin{align*}
        w_{\leq F}j_{*}EM^{F}_{U} \rightarrow EM^{F}_{X} \rightarrow w_{\leq F}j_{*}w_{>F}j'_{*}1_{V} \rightarrow
    \end{align*}
    By left exactness of $j_{*}$, third term vanishes so we get an isomorphism $EM^{F}_{X} \cong  w_{\leq F}j_{*}EM^{F}_{U} $.
\end{proof}
\begin{lemma}\label{pullback of EM to U is constant sheaf}
    For any open dense immersion $j:U \hookrightarrow X$ with $U$ regular, $j^{*}EM^{F}_{X} \cong 1_{U}$.
\end{lemma}
\begin{proof}
    First, we show that $j^{*}$ is weight exact i.e 
    \begin{align*}
    j^{*}(^{w}DM^{\leq F}(X)) \subset \;^{w}DM^{\leq F}(U)\;\; \text{and}\;\; j^{*}(^{w}DM^{>F}(X)) \subset \;^{w}DM^{>F}(U). 
    \end{align*}
   Let $M \in\;^{w}DM^{\leq F}(X)$ and $s: \Spec K \hookrightarrow U$ be any point in $U$. Then, $s^{*}j^{*}M= (s')^{*}M$ where $s':=j\circ s: \Spec K \hookrightarrow X$. Since $(s')^{*}M \in\;^{w}DM^{\leq F(t_{K})-t_{K}}(\Spec K) $, it follows from the gluing definition that $j^{*}M \in\;^{w}DM^{\leq F}(U)$. The proof of the second inclusion is similar, using the fact that $j^{*} \cong j^{!}$.\\
   Next, we claim that $j^{*}$ commutes with truncation functors $w_{\leq F}$ and $w_{>F}$. Let $M \in DM(X)$. Since $j^{*}$ is a triangulated functor, applying it to the decomposition triangle of $M$ yields:
   \begin{align*}
       j^{*}w_{\leq F}M \rightarrow j^{*} M \rightarrow j^{*}w_{>F}M \rightarrow
   \end{align*}
   By weight-exactness of $j^{*}$, the first term lies in $^{w}DM^{\leq F}(U)$ and the third in $^{w}DM^{>F}(U)$. So, by the uniqueness of the decomposition triangle, we conclude that 
   \begin{align*}
   w_{\leq F}j^{*}M \cong j^{*}w_{\leq F}M\;\;\; \text{and}\;\;\;  w_{>F}j^{*}M \cong j^{*}w_{> F}M .
   \end{align*}
   Now recall that $EM^{F}_{X} = w_{\leq F}j_{*}1_{U}$ for some regular open dense immersion $j:U \hookrightarrow X$. Applying $j^{*}$, we obtain:
   \begin{align*}
       j^{*}EM^{F}_{X} = j^{*}w_{\leq F}j_{*}1_{U} \cong w_{\leq F}j^{*}j_{*}1_{U}.
   \end{align*}
 Finally,  since $j^{*}j_{*}1_{U} \cong 1_{U}$ and $1_{U} \in \;^{w}DM^{\leq F}(U)$, it follows that $j^{*}EM^{F}_{X} \cong 1_{U}$.
\end{proof}
\begin{proposition}\label{stability of F under finite map}Let $\pi: X' \rightarrow X$ be finite map of irreducible schemes of dimension 3. Then the following is true:
\begin{align*}
    \pi_{*}(^{w}DM^{\leq F}(X')) \subseteq\; ^{w}DM^{\leq F}(X)\\
    \pi_{*}(^{w}DM^{> F}(X')) \subseteq\; ^{w}DM^{> F}(X)
    \end{align*}
\end{proposition}
 \begin{proof}
       First, we prove this for $\pi$ a finite etale morphism. Let $M \in\; ^{w}DM^{\leq F}(X')$ and let $s: \Spec k(s) \hookrightarrow X$ be a point $s$ with the residue field $k(s)$. Let  $X'_{s} \xhookrightarrow{s'}X'$ the fibre at $s$. We need to show that $s^{*}\pi_{*}M=\pi_{*}s'^{*}M \in\;^{w}DM^{\leq F}(\Spec k(s))$. The fibre $X'_{s}$ is disjoint union of $\Spec k'$ where $k'$ is finite separable extension of $k(s)$ and hence each point is both open and closed, and $s'^{*}M\cong \oplus_{\#} s_{i*}s^{i*}M$ a finite direct sum. Hence if  $t_{i}: \Spec k'\xhookrightarrow{s_{i}} X'_{s}\rightarrow  \Spec \kappa(s)$  denotes the composite, 
       \[
       	\pi_{*}s'^{\ast}M \cong \bigoplus_{i} \pi_{\ast} s_{i\ast} t^{\ast}M \cong \bigoplus_{i} t_{i*}t_{i}^{*}M \in {^{w}DM^{\le F}}(\Spec \kappa (s_{0})
       \]
        Now $\kappa (s_{i})$ is a finite separable extension of $k$, hence $t_{i\ast }t_{i}^{\ast }$ is just multiplication by degree of this extension and we are done. 
       
      Next, we prove the statement for $\pi$ as stated in the hypothesis. Let $X'_{sm}$ be a smooth locus of $X'$. Since $\pi$ is proper, $\pi(X'-X'_{sm})$ is a closed subset in $X$. Then, $U:= X- \pi(X'-X'_{sm}) \xhookrightarrow{j} X$ is an open immersion. Take $U':= \pi^{-1}(U)$. Since $U'$ is regular, by generic smoothness we can restrict $U$ and assume that $\pi|_{U}:U' \rightarrow U$ is an etale morphism. Then, $j^{*}\pi_{*}M \in\; ^{w}DM^{\leq F}(U)$ by above. Also, $i^{*}\pi_{*}M = \pi_{*}i^{*}M$ is in $^{w}DM^{\leq F}$ by Noetherian induction where $i:X-U\hookrightarrow U$ is the closed complement, and since $i^{*}M \in\;^{w}DM^{\leq F}(X-U)$ by definition.
     
     The second equation is similar. 
       \end{proof}
       \begin{corollary}\label{normalization}
    Let $\pi:\Tilde{X} \rightarrow X$ be normalization of a irreducible 3-fold $X$. Then, $\pi_{*}EM^{F}_{\Tilde{X}} \cong EM^{F}_{X}$. 
\end{corollary}
\begin{proof}
    There is a smooth open set $U \xhookrightarrow{j}X$ such that the fibre product $\Tilde{U}$ is regular and $\pi|_{U}: \Tilde{U} \rightarrow U$ is an isomorphism. Denote inclusion $\Tilde{U} \hookrightarrow \Tilde{X}$ by $\Tilde{j}$. So, we have the following triangle:
    \begin{align*}
        EM^{F}_{\Tilde{X}} \rightarrow \Tilde{j}_{*}1_{\Tilde{U}} \rightarrow w_{>F}\Tilde{j}_{*}1_{\Tilde{U}} \rightarrow
    \end{align*}
    Apply $\pi_{*}$ to the above triangle:
    \begin{align*}
        \pi_{*}EM^{F}_{\Tilde{X}} \rightarrow j_{*}1_{U} \rightarrow \pi_{*}w_{>F}\Tilde{j}_{*}1_{\Tilde{U}} \rightarrow
    \end{align*}
    Now, using Proposition \ref{stability of F under finite map}, we get that $\pi_{*}EM^{F}_{\Tilde{X}} \cong w_{\leq F}j_{*}1_{U} \cong EM^{F}_{X}$.
\end{proof}
\begin{definition}
Let $X$ be an irreducible variety of dimension $3$. Let $X=U \bigsqcup Z$ be its stratification such that $Z$ is a closed subset of dimension zero; that is, $Z$ consists of finitely many closed points of $X$. For each closed point $s:\Spec K \hookrightarrow X$ of $Z$, there is a t-structure $(^{w}DM^{\leq 2}(K), \;^{w}DM^{>2}(K))$ on the category $DM^{coh}_{2}(K)$. Since the four functors on $DM^{coh}_{2}(-)$ satisfy the formalism of gluing, using \ref{ordinaryGluing} inductively, we obtain a glued t-structure on $DM^{coh}_{2}(Z)$, which we denote by $\left(^{w}DM^{\leq 2}(Z), ^{w}DM^{>2}(Z)\right)$.

We then define the t-structure $ \left(^{w}DM^{\leq (\infty,2)}(X),\;^{w}DM^{> (\infty,2)}(X)\right)$
as the one obtained by gluing the $\infty$ on $DM^{coh}_{3,dom}(U)$ with $\left(^{w}DM^{\leq 2}(Z), ^{w}DM^{>2}(Z)\right)$ on $DM^{coh}_{2}(Z)$.
 \end{definition}
 \begin{lemma}\label{relation of F and (infty, 2)} 
 Let $X$ be an irreducible variety of dimension $3$ with stratification $X=U \bigsqcup Z$ such that $Z$ is a closed subset of dimension zero. Let  $\left(^{w}DM^{\leq (\infty,2)}(X),\;^{w}DM^{> (\infty,2)}(X)\right)$ be the t-structure defined as above. Then the following inclusion holds:
\begin{align*}
    ^{w}DM^{\leq F}(X) \subset\; ^{w}DM^{\leq (\infty,2)}(X)
    \end{align*}
\end{lemma}
\begin{proof}
By the definition of the glued t-structure, we have
    \begin{align*}
    ^{w}DM^{\leq (\infty,2)}(X)=\{M \in DM^{coh}_{3,dom}(X) \mid s^{*}M \in\;^{w}DM^{\leq 2}(K)\;\forall\; s:\Spec K \in Z\}
  \end{align*}
Suppose $M \in\;^{w}DM^{\leq F}(X)$. Then, by the definition of $^{w}DM^{\leq F}(X)$ (see Theorem \ref{def of leq F}), and using the fact that the points of $Z$ are closed points of $X$, we obtain that $s^{*} M \in\;^{w}DM^{\leq 2}(K)$ for all $s \in Z$. Thus, $M \in\;^{w}DM^{\leq (\infty,2)}(X)$.
\end{proof}
\begin{corollary}\label{corollary of relation btw F and (infty,2)}
Let $X$ be as in the previous lemma. Then, $EM^{F}_{X}\cong
w_{\leq(\infty,2)}\,j_{*}EM^{F}_{U}$
\end{corollary}
\begin{proof}
From Lemmas \ref{EM(X) in terms of EM(U)} and \ref{Pre on t-structures}, we have the isomorphism
\[ EM^{F}_{X}\cong w_{\leq F}\,j_{*}EM^{F}_{U}  \cong w_{\leq F}\,w_{\leq (\infty,2)}\,j_{*}EM^{F}_{U}.\]
Since $j^{*}$ commutes with $w_{\leq (\infty,2)}$ (see Lemma \ref{j* commutes with truncation functor}), we have that $j^{*}w_{\leq (\infty,2)}\,j_{*}EM^{F}_{U} \cong EM^{F}_{U}$. Therefore, for every point $s:\Spec K \hookrightarrow X$ lying in $U$, we have 
\[s^{*}w_{\leq (\infty,2)}\,j_{*}EM^{F}_{U} \in\;^{w}DM^{\leq F(t_{K})-t_{K}}(K).\] 
For any point $s$ in $Z$,
\[s^{*}w_{\leq (\infty,2)}\,j_{*}EM^{F}_{U} \in \;^{w}DM^{\leq 2}(K)\]
follows clearly from the definition of $^{w}DM^{\leq (\infty,2)}(X)$.
Thus, we conclude that
\[w_{\leq (\infty,2)}\,j_{*}EM^{F}_{U} \in\;^{w} DM^{\leq F}(X),\]  
which yields the desired isomorphism.
\end{proof}
\begin{corollary}\label{s upper star and shriek acting on EM}
    Let $X$ be as in the previous lemma, and let $i: Z \hookrightarrow X$ denote the closed immersion. Then, for all $s$ in $Z$,
    \begin{align*}
        s^{*}EM^{F}_{X} \cong w_{\leq 2}\,s^{*}j_{*}\,EM^{F}_{U}\;\;\;\text{and}\;\;\;s^{!}EM^{F}_{X} \cong w_{>2}\, s^{*}j_{*}\,EM^{F}_{U}[-1].
    \end{align*}
\end{corollary}
\begin{proof}
We know that
\[EM^{F}_{X} \cong w_{\leq (\infty,2)}\,j_{*}EM^{F}_{U}.\] 
Since $s$ is open in $Z$, using Lemma \ref{j* commutes with truncation functor}, it follows that 
\[ s^{*}EM^{F}_{X} \cong w_{\leq 2}\,s^{*}j_{*}\,EM^{F}_{U}.\]
For the second isomorphism, apply the functor $i^{!}$ to the gluing triangle \ref{gluing triangle}. As $i^{!}j_{*}=0$, we obtain:
   \[\begin{aligned}
        i^{!}\, w_{\leq (\infty,2)}\, j_{*}\,EM^{F}_{U} \rightarrow 0 \rightarrow w_{>2}\, i^{*} j_{*}\,EM^{F}_{U} \rightarrow
     \end{aligned}\]
     Therefore, \[i^{!}\, w_{\leq (\infty,2)}\, j_{*} EM^{F}_{U} \cong w_{>2}\, i^{*} j_{*}\,EM^{F}_{U}[-1].\]
Since $s$ is open in $Z$ and $j^{*}$ commutes with $w_{\leq (\infty,2)}$, we deduce that
\[
s^{!}EM^{F}_{X} \cong w_{>2}\, s^{*}j_{*}\,EM^{F}_{U}[-1].
\]
\end{proof}
\subsection{Refinement of Morel's truncations, motivically}\label{refined in motivic setting}
\label{refinement motivically}
\para
Let $X$ be smooth and projective over a perfect field $k$. There is a split short exact sequence
\begin{align*}
    0 \rightarrow Pic^{0}(X)_{\mathbb{Q}} \rightarrow Pic(X)_{\mathbb{Q}} \rightarrow NS(X)_{\mathbb{Q}} \rightarrow 0
\end{align*} where $Pic(X)$ denotes the Picard group, while $Pic^{0}(X)$ denotes the subgroup of algebraically trivial elements. Here, $NS(X)$ is the Neron-Severi group of $X$. It is a finitely generated group. Fix a splitting $f: NS(X)_{\mathbb{Q}} \rightarrow Pic(X)_{\mathbb{Q}}$. We know that 
	\[
	Pic(X)_{\mathbb{Q}} = \Hom_{DM(k)} \left(1_{k}(-1)[-2], h(X)\right).
	\]
	For each generator $x_{i}$ of $NS(X)$, consider the morphism $f_{i}': 1_{k}(-1)[-2] \rightarrow h(X)$ on RHS associated to the element $f(x_{i})$ on LHS. Since (see 4.6 of \cite{scholl1994classical})
    \begin{align*}Pic^{0}(X)_{\mathbb{Q}}= \Hom\left(1_{k}(-1)[-2], h^{1}(X)\right)\;\; \text{and}\; \;\Hom \left(1_{k}(-1)[-2],h^{0}(X)\right)=0,
    \end{align*}
    so we obtain that 
    \[NS(X)_{\mathbb{Q}} = \Hom \left(1_{k}(-1)[-2], h^{>1}(X)\right). \]
    As a result, each map $f_{i}'$ must factor through a map $f_{i}:1_{k}(-1)[-2] \rightarrow h^{>1}(X)$. Together, these $f_{i}$'s define a map 
	\[
		g_{X}:=(f_{1},\dots,f_{\alpha}): \bigoplus \limits_{i=1}^{\alpha} 1_{k}(-1)[-2] \longrightarrow h^{>1}(X).
	\]
Let $C_{X}$ denote the cone of the map $g_{X}$ in $DM(k)$. 
\subsubsection{} Recall that
\begin{align*}
DM^{coh}(k):= \langle p_{*}1_{X} \mid \;p: X \rightarrow k \;\text{smooth and projective} \rangle 
\end{align*} where the generation is in $DM(k)$. Now, we define the following collection of objects and corresponding triangulated subcategories:
\begin{align*}
M^{\leq 1}(k) &:=\{h^{\leq 1} (X) \mid X\; \text{is smooth, projective over}\; k\}\\
M^{alg}(k) &:=\{h(k')(-1)[-2] \mid  k'/k \;\text{finite, separable extension}\}\\
M^{\leq 1+alg}(k) &:= M^{\leq 1}(k) \cup M^{alg}(k)\\
M^{> 1+alg}(k) &:= \{C_{X} \mid X\; \text{is smooth, projective over}\;k\}\\
^{w}DM^{\leq 1+alg}(k)&:= \langle M^{\leq 1+alg}(k)\rangle \hspace{1cm} ^{w}DM^{> 1+alg}(k) := \langle M^{> 1+alg}(k) \rangle
\end{align*}
\begin{theorem}\label{1+alg for DM-coh(k)}
Let $A \in M^{\leq 1+alg}$ and $B \in M^{> 1+alg}$. Then $\Hom(A,B[m])=0\; \forall\, m \in \mathbb{Z}$. 
\end{theorem}
\begin{proof}
First, assume $A= h^{\leq 1}(X)$ and $B= C_{Y}$. Consider the triangle:
\begin{align}\label{Main eq of 1+alg}
    \bigoplus\limits_{\alpha} 1_{k}(-1)[-2] \xrightarrow{g_{Y}} h^{>1}(Y) \rightarrow C_{Y} \rightarrow
\end{align}
From the definition of $^{w}DM^{>1}(\Spec k)$ and the fact that $1_{k}(-1)[-2] \cong h^{2}(\mathbb{P}^{1})$, we have that $1_{k}(-1)[-2]$ and $h^{>1}(X) \in\; ^{w}DM^{>1}(\Spec k)$. Therefore, $C_{Y} \in\; ^{w}DM^{>1}(\Spec k)$, the latter being a triangulated category. Since $A=h^{\leq 1}(X) \in\; ^{w}DM^{\leq 1}(\Spec k)$, we get $\Hom(h^{\leq 1}(X), C_{Y}[m])=0$.

Now, assume $A=h(k')(-1)[-2]$ where $k'/k$ is finite, separable extension and $B=C_{Y}$. We claim that
\[
	\Hom\left(h(k')(-1)[-2][-m], \bigoplus\limits_{\alpha}1_{k}(-1)[-2]\right) \rightarrow \Hom\left(h(k')(-1)[-2],h^{>1}(Y)\right)
\] 
induced by $g_{Y}$ is an isomorphism for all $m$. First assume $m=0$. Since $h(k')^{\vee} \cong h(k')$, we have
\begin{align*}\Hom(h(k')(-1)[-2], 1_{k}(-1)[-2])= \Hom(1_{k},h(k')) \cong \mathbb{Q}.
\end{align*} and more generally, 
\begin{align*}\Hom\left(h(k')(-1)[-2], h^{>1}(Y)\right)&= \Hom\left(1_{k}(-1)[-2], h^{>1}(Y \cross \Spec k')\right) 
&\cong NS\left(Y \cross \Spec k'\right)_{\mathbb{Q}}.
\end{align*} Then the map $g_{Y}$ becomes $(r_{1},\dots, r_{\alpha}) \mapsto r_{1}f_{1}+\dots + r_{\alpha}f_{\alpha}$. Clearly, this is an isomorphism of $\mathbb{Q}$- vector spaces. Next, by motivic cohomology vanishings observe that for $m \neq 0$,
\begin{align*}
\Hom\left(h(k')(-1)[-2], \bigoplus\limits_{\alpha}1_{k}(-1)[-2][m]\right) \cong \bigoplus\limits_{\alpha}\Hom\left(1_{k},h(\Spec k')[m]\right)=0
\end{align*}
\text{and for $m \neq -1,0$}
\begin{align*}
\Hom\left(h(k')(-1)[-2], h^{>1}(Y)[m]\right)= \Hom\left(1_{k},h^{>1}(Y \cross \Spec k')(1)[m+2]\right)=0.
\end{align*}
We have \[\Hom\left(1_{k}(-1)[-2], h^{0}(Y)[-1]\right) \cong \mathcal{O}^{*}(L) \tensor \mathbb{Q}\]
where $Y \rightarrow \Spec L \rightarrow \Spec k$ is the Stein factorization and \[\Hom\left(1_{k}(-1)[-2], h(Y)[-1]\right) \cong \mathcal{O}^{*}(Y) \tensor \mathbb{Q}.\] Since $L \cong \mathcal{O}(Y)$ and $h^{0}(Y) \bigoplus h^{>1}(Y)$ is the summand of $h(Y)$, it follows that even for $m=-1$ \[\Hom\left(h(k')(-1)[-2], h^{>1}(Y)[-1]\right)=0.\] Now, writing the long exact sequence corresponding to the functor $\Hom\left(h(k')(-1)[-2], -\right)$, we obtain that $\Hom\left(h(k')(-1)[-2], C_{Y}[m]\right)=0 \; \forall \;m \in \mathbb{Z}$. 
\end{proof}
\begin{corollary}\label{1+alg t-structure}
The pair $(^{w}DM^{\leq 1+alg}(k),\; ^{w}DM^{>1+alg}(k))$ forms a t-structure on $DM^{coh}(k)$. 
\end{corollary}
\begin{proof}
For any object $h(X) \in DM^{coh}(k)$, 
\begin{align*}
h^{\leq 1}(X) \bigoplus\left(\bigoplus\limits_{\alpha} 1_{k}(-1)[-2]\right) \rightarrow h(X) \rightarrow C_{X} \rightarrow
\end{align*}
is the decomposition triangle. Therefore, we are done by \ref{1+alg for DM-coh(k)} and \ref{tFromGenerators}.
\end{proof} 
\begin{notation}
	We will denote the truncations for the above $t$-structure on $DM^{coh}(X)$ above as $w_{\le 1+alg}$ and $w_{>1+alg}$. 
\end{notation}
\begin{remark}
Ayoub-Viale \cite[\S 2]{ayoub20091motivic} have shown that the inclusion $DM^{coh}_{1-mot}(k) \hookrightarrow DM^{coh}(k)$ admits a right adjoint, where $DM^{coh}_{1-mot}(k):=\{p_{*}1_{X}\mid p:X \rightarrow k \;\text{proper}\;, \dim X \leq 1\}$. This leads to a t-structure $t^{1}_{M}$ whose negative part is $DM^{coh}_{1-mot}(k)$. It is easy to see that the t-structure $(^{w}DM^{\leq 1+alg}(k),\;^{w}DM^{>1+alg}(k))$ coincides with $t^{1}_{M}$, but our construction offers a more palatable formulation of Ayoub-Viale’s method.
\end{remark}
\begin{remark}\label{twisted refined functor}
We will also have the occassion of using a twisted version of this $t$-structure -- recall that $A\mapsto A(-1)$ is an auto equivalence of the category $DM(k)$. 

Thus, the full subcategories $({^{w}DM^{\le 1+alg}(X)}(-1), {^{w}DM^{> 1+alg}(X)}(-1))$ would form a $t$-structure on the twisted subcategory $DM^{coh}(X)(-1)$. We will denote the corresponding truncations by $w_{\le 1+alg (-1)}$ and $w_{> 1+alg (-1)}$.
\end{remark}
\begin{proposition}\label{leq 1+alg is in conservative part}
      The category $DM^{\leq 1+alg}(k) \subset DM^{Ab}(k) = \langle h(A)\mid A/k\text{ is an abelian variety}\rangle$.
        \end{proposition}
        \begin{proof}
           For any $X$ of dimesnion $d$,\; $h^{0}(X) \in DM^{Ab}(k)$. Thus, $h^{2d}(X) \cong h^{0}(X)(-d) \in DM^{Ab}(k)$, as $DM^{Ab}(k)$ is closed under tate twists. To show $h^{1}(X) \in DM^{Ab}(k)$, note that $h^{1}(X)$ is the summand of $h^{1}(C)$ for some curve $C$ (see \cite[4.6 Remark (i)]{scholl1994classical}) and, $h^{1}(C) \cong h^{1}(J(C))$ where $J(C)$ denote the Jacobian variety associated to $C$. Since $J(C)$ is an abelian variety, it follows that $h^{1}(X) \in DM^{Ab}(k)$.
           \end{proof}
\begin{proposition}\label{1+alg for surfaces}
    For a smooth projective surface $Y$, the triangle \[ \bigoplus\limits_{\alpha}1_{k}(-1)[-2] \xrightarrow{g_{Y}} h^{>1}(Y) \rightarrow C_{Y} \rightarrow \] splits. In particular, we have the decomposition of $h^{2}(Y)$ in $CHM(k)$.
    \end{proposition}
\begin{proof}
   First, note that map $g_{Y}$ factors through inclusion $h^{2}(Y)[-2] \hookrightarrow h^{>1}(Y)$ since $\Hom\left(\mathbb{L},h^{2}(Y)[-2]\right)= NS(Y)_{\mathbb{Q}}$. So we may assume $g_{Y}: \bigoplus\limits_{\alpha}1_{k}(-1)[-2] \rightarrow h^{2}(Y)[-2]$ and $C_{Y}$ to be its cone. To show this triangle splits, it is enough to show that $\Hom\left(C_{Y}, \mathbb{L}[1]\right)=0$. Consider the long exact sequence corresponding to the functor $\Hom\left(-, \mathbb{L}[1]\right)$:
     \begin{align*}
     \hspace{-1.3cm}
    \cdots \rightarrow \Hom\left(h^{2}(Y)[-2][1], \mathbb{L}[1]\right) \rightarrow \Hom\left(\bigoplus\limits_{\alpha}\mathbb{L}[1], \mathbb{L}[1]\right) &\rightarrow
     \Hom \left(C_{Y}, \mathbb{L}[1]\right) \rightarrow \Hom \left(h^{2}(Y)[-2], \mathbb{L}[1]\right) \rightarrow \cdots 
     \end{align*}
We know $h(Y)^{\vee} = h(Y)(2)[4]$, so  
\begin{align*}
    \Hom\left(h(Y), \mathbb{L}[1]\right)&=\Hom\left(1_{k}, h(Y)^{\vee}(-1)[-1]\right) = \Hom\left(1_{k},h(Y)(1)[3]\right)=0
\end{align*}
Since, $\Hom\left(h^{2}(Y)[-2],\mathbb{L}[1]\right) \hookrightarrow  \Hom\left(h(Y), \mathbb{L}[1]\right)$, we have  $\Hom\left(h^{2}(Y)[-2],\mathbb{L}[1]\right)=0$.
Next, see that, by motivic cohomology, \[\Hom\left(\mathbb{L}[1], \mathbb{L}[1]\right)= \Hom\left(1_{k}, 1_{k}\right)= H^{0,0}(\Spec k)=\mathbb{Q}\]
and since ${(h^{2}(Y)[-2])}^{\vee}=(h^{2}(Y)[-2]) \tensor \mathbb{L}^{-2}$ we have,
\begin{align*}\Hom\left(h^{2}(Y)[-2][1], \mathbb{L}[1]\right)&=\Hom\left(\mathbb{L}^{-1}, {(h^{2}(Y)[-2])}^{\vee}\right)
=\Hom\left(\mathbb{L},h^{2}(Y)[-2]\right)= NS(Y)_{\mathbb{Q}}.
\end{align*}
Thus, the map $\Hom\left(g_{Y}[1],\mathbb{L}[1]\right)$ becomes $r_{1}f_{1}+ \cdots +r_{\alpha}f_{\alpha} \mapsto (r_{1}, \cdots ,r_{\alpha})$, which is an isomorphism of $\mathbb{Q}$-vector spaces. From the above calculations, it follows that $\Hom(C_{Y}, \mathbb{L}[1])=0$. Therefore, we obtain a decomposition $h^{2}(Y)[-2]= \left(\oplus_{\alpha} 1_{k}(-1)[-2] \right) \bigoplus C_{Y}$.
\end{proof}
\begin{remark} The above calculation reproves the key result of \cite{murre_kahn_pedrini}, giving the refined Chow–Künneth decomposition of the motive of a surface. Therefore in terms of refined truncations, for any smooth projective surface $Y$, we have that
 \[ h^{2}_{alg}(Y)= w_{\leq 1+alg}\, w_{\geq 2}\,h(Y) \;\;\text{and} \;\;h^{2}_{tr}(Y)= w_{> 1+alg}\,w_{\leq 2}\,h(Y).\] 
\end{remark}
We record the following for later use: 
\begin{proposition}\label{objects in 1+alg}For any $Y \in SmProj/k$, the following are true:
\begin{align*}
	h(Y)(-j) \in\; ^{w}DM^{>1+alg}(k)\text{ for }j \geq 2& & h^{\geq 1}(Y)(-1) \in \;^{w}DM^{>1+alg}(k).
\end{align*}
\end{proposition}
\begin{proof}
    For the first claim, it is enough to observe that $h(Y)(-j)$ is a summand of $h^{\ge 2j}(Y\times \mathbb P^{r}(k))$ which is in $^{w}DM^{>2}(k)\subset {^{w}DM^{>1+alg}(k)}$ by definition. For the second claim, note that $h^{\ge 1}(Y)(-1)$ is a summand of $h^{\ge 3}(Y\times \mathbb P^{1}(k))$ and again, $^{w}DM^{\ge 3}(k)\subset {^{w}DM^{>1+alg}(k)}$.
\end{proof}
\subsection{Refining Morel's truncations in mixed categories.}

\begin{para}
	It is possible and useful to define an analogue of the above refinement in mixed categories, this is what we do next. 
	Let $k$ be a field of characteristic zero with a fixed embedding $\sigma: k \hookrightarrow \mathbb{C}$. For a variety $X$ over $k$, let $X^{an}$ be the associated complex analytic space. Then the singular cohomology of $X^{an}$ (with $\Q$ coefficients) has a mixed Hodge structure. Let $MHS_{\Q}(k)$ denote the category of natural mixed Hodge structures over $k$. Recall that we have a mixed realization functor
\[ R_{\mathcal{M}}: DM(k) \rightarrow D^{b}(MHS_{\Q}(k)).\]
It takes motive $h(X)$ of a smooth projective variety $p:X \rightarrow k$ to a complex $R(p_{*}1_{X})$ of mixed Hodge structures computing the singular cohomology of $X^{an}$.
\end{para}
\begin{para}
	Recall that the mixed category of $\Q$-Hodge structures $MHS_{\Q}(k)$ is abelian, Noetherian, and Artinian, with simple objects also pure. Let $HS(n,\Q)$ be the category of pure Hodge structures of weight $n$ which are also simple. Then the sub category $\bigcup_{n} HS(n,\Q)$ generates the triangulated category $D^{b}(MHS_{\Q}(k))$.
\end{para}

\begin{proposition}\label{1+alg in real thm} The following subcategories in $D^{b}(MHS_{\Q}(k))$ form a $t$-structure:
\begin{align*}
^{w}D^{\leq 1+alg}(k)&:=\langle \{V \in HS(n, \mathbb{Q}) \mid n \leq 1\} \cup \{\mathbb{Q}(-1)\}\rangle\\
^{w}D^{>1+alg}(k)&:=\langle \{V \in HS(n, \mathbb{Q}) \mid n >2\} \;\cup \{V \in HS(2, \mathbb{Q}) \mid V\;\text{is simple and}\;V \neq \mathbb{Q}(-1) \}\rangle
\end{align*}
\end{proposition}
\begin{proof}
There is an analogue of Morel's t-structure on $D^{b}(MHS(k))$. In particular for $V \in HS(n,\mathbb{Q})$ and $V' \in HS(n',\mathbb{Q})$ we have the vanishings
\begin{align*}
\Hom(V,V'[m])=0 \hspace{1cm} \text{for}\; n \leq 1,\; n'>2 \;\text{and}\; m \in \mathbb{Z}
\end{align*}
Let $V' \in HS(2,\mathbb{Q})$. In $D^{b}(MHS(k))$, the extensions $Ext^{i}(N,N')=0$ for $N$ of weight $\leq n$ and $N'$ of weight $\geq n'$ if $n \leq n'+i$. So, $\Hom(\mathbb{Q}(-1),V'[m])=0\;\text{for}\;m>0$.  The standard t-structure on the derived category $D^{b}(MHS(k))$ gives $\Hom(V,V'[m])=0$ for $m <0$. Since both $\mathbb{Q}(-1)$ and $V$ are simple objects and $V \neq \mathbb{Q}(-1)$, we have $\Hom(\mathbb{Q}(-1),V)=0.$

Therefore, to show that it is a $t$-structure, by \ref{tFromGenerators} it is enough to construct a decomposition triangle for each object $V$ in $HS(n,\Q)$. The corresponding triangles are 
\begin{align*}
	V\rightarrow V\rightarrow 0\rightarrow \text{ for }n\le 1\text{ or }V\cong \Q(-1)& &0\rightarrow V\rightarrow V\rightarrow \text{ for }n\ge 2\text{ and }V\not\cong \Q(-1)
\end{align*}
and we are done. 
\end{proof}

Finally, we connect the motivic and the mixed constructions through the realization functors:

\begin{proposition}\label{mixed realisation commute with 1+alg}
$R_{\mathcal{M}}\left(^{w}DM^{\leq 1+alg}(k)\right)\subset \;^{w}D^{\leq 1+alg}(k)$ and $R_{\mathcal{M}}\left(^{w}DM^{> 1+alg}(k)\right) \subset \;^{w}D^{>1+alg}(k).$
\end{proposition}
\begin{proof}
    It is enough to check on generators because $R_{\mathcal{M}}$ is a triangulated functor. Note that $R_{\mathcal{M}}(h^{i}(Y))= H^{i}(R_{\mathcal{M}}(p_{*}1_{Y}))= H^{i}(Y^{an},\mathbb{Q})$ is a pure Hodge structure of weight $i$. Also, $R_{\mathcal{M}}(1_{k}(-1))=\mathbb{Q}(-1)$. Thus, $R_{\mathcal{M}}(^{w}DM^{\leq 1+alg}(k)) \subset\;^{w}D^{\leq 1+alg}(k)$. Next, to show $R_{\mathcal{M}}(C_{Y}) \in \;^{w}D^{>1+alg}(k)$, it suffices to show that 
\[\Hom(V, R_{\mathcal{M}}(C_{Y}))=0\;\;\; \forall\; V \in\;^{w}D^{\leq 1+alg}(k).\] 
Since $C_{Y} \in\;^{w} DM^{>1}(k)$, so $R_{\mathcal{M}}(C_{Y}) \in \;^{w}D^{>1}(k)$. This implies that $\Hom(V,R_{\mathcal{M}}(C_{Y}))=0$ for $V \in HS(n, \mathbb{Q})$ with $n \leq 1$.

Now taking the realization of triangle (\ref{Main eq of 1+alg}) and then applying the functor $\Hom(\mathbb{Q}(-1), -)$ to it, we get the following long exact sequence.
\begin{align*}\cdots\rightarrow\Hom(\mathbb{Q}(-1), \bigoplus\mathbb{Q}(-1)) \rightarrow \Hom\left(\mathbb{Q}(-1), H^{\geq 2}(Y^{an})\right) \rightarrow \Hom\left(\mathbb{Q}(-1),R_{\mathcal{M}}(C_{Y})\right) \rightarrow \\\Hom\left(\mathbb{Q}(-1), \bigoplus \mathbb{Q}(-1)[1]\right) \rightarrow \cdots
\end{align*}
Due to weight reasons, the last term vanishes and the middle term is equal to $\Hom\left(\mathbb{Q}(-1), H^{2}(Y^{an})\right)$. We claim that the first map is an isomorphism, which would then imply $\Hom\left(\mathbb{Q}(-1),R_{\mathcal{M}}(C_{Y})\right)=0$. To prove the claim, we use the Hodge-Lefschetz theorem for divisors. Observe that it is equivalent \cite{MR1273954} to 
\[ NS(Y)_{\mathbb{Q}} \cong \Hom_{HS}(\mathbb{Q}(-1), H^{2}(Y^{an})).\]
We proceed as in Section \ref{refined in motivic setting}. For each generator $x_{i}$ of $NS(Y)_{\mathbb{Q}}$, consider the map $f_{i}: \mathbb{Q}(-1) \rightarrow H^{2}(Y^{an})$. Together, they define a map \[g_{M}:=(f_{1}, \cdots, f_{\alpha}): \bigoplus\limits_{i=1}^{\alpha} \mathbb{Q}(-1) \rightarrow H^{2}(Y^{an}). \]
Notice that $R_{\mathcal{M}}(g_{Y})=g_{M}$. Then the map
\[ \Hom(\mathbb{Q}(-1), \bigoplus\mathbb{Q}(-1)) \rightarrow \Hom\left(\mathbb{Q}(-1), H^{\geq 2}(Y^{an})\right) \] induced by $g_{M}$ becomes $r_{1}+ \cdots +r_{\alpha} \mapsto r_{1}f_{1}+ \cdots r_{\alpha}f_{\alpha}$, which is isomorphism of $\mathbb{Q}$-vector spaces.
\end{proof}

\subsection{Geometry} 
\begin{definition}
    A closed subscheme $Y$ of an irreducible scheme $X$ is said to be a \emph{simple normal subvariety} if  $Y = \bigcup_{i\in I}Y_{i}$ with $Y_{i}$ irreducible components of $Y$ and each $r$-fold intersection $Y^{J}_{(r)}:= \bigcap\limits_{j \in J} Y_{j}$ is regular for all $\phi \neq J \subset I$ with $|J|=r+1$. 
    
    This is said to be a \emph{simple normal crossing divisor} if in addition $\dim Y^{J}_{(r)}= \dim X- (r+1)$ whenever $Y^{J}_{(r)} \neq \phi$. We denote by $Y_{(r)}$ the union of $r$-fold intersections of irreducible components of $Y$.
\end{definition}
\begin{proposition}\label{SNC divisor}
Let $X \rightarrow \Spec\;k$ be a smooth projective variety over $k$ of dimension 3 and $s: F \hookrightarrow X$ be a closed immersion such that $F$ is a simple normal crossing divisor. Let $F'$ be the disjoint union of the irreducible components of $F$, and let $q: F' \rightarrow F$ be the morphism induced by the inclusions of the components into $F$. Then, 
    \begin{enumerate}[(i)]
    \item $w_{\leq 2}\pi_{*}s^{!}1_{X} \cong h^{0}(F')(-1)[-2].$
    \item There exist $n \in \mathbb{Z}$ such that $w_{> 2}\pi_{*}s^{!}1_{X} \in Ext^{n}(A[\mathbb{Z}_{\geq 0}])$ where \[A=\{h(F_{(r)}')(-r-1)[-2r-2],h^{\geq 1}(F')(-1)[-2]\mid r \geq 1\}\] and $F_{(r)}'$ is the scheme-theoretic inverse image of $F_{(r)}$ along the map $q$.
    \end{enumerate}
\end{proposition}
\begin{proof}
 \textit{(i)} First, suppose that $F$ is regular. Then, by purity,
    \[w_{\leq 2}\pi_{*}s^{!}1_{X} \cong w_{\leq 2}\pi_{*}1_{F}(-1)[-2] \cong w_{\leq 2}h(F)(-1)[-2]= h^{0}(F)(-1)[-2].\]
    Now, assume $F$ is not regular. Consider the following closed and open immersions, respectively:
\[i_{1}: F_{(1)} \hookrightarrow F \;\;\;\text{and}\;\;j_{1}:U_{1}:= F-F_{(1)} \hookrightarrow F\]
The localization triangle for $s^{!}1_{X}$ on $F$ gives us
    \begin{align}\label{A}
{i_{1}}_{*}{i_{1}}^{!}s^{!}1_{X} \rightarrow s^{!}1_{X} \rightarrow {j_{1}}_{*}{j_{1}}^{!}s^{!}1_{X} \rightarrow
    \end{align}
    Observe that we have the factorization of $s \circ j_{1}:F-F_{(1)} \xhookrightarrow{i} X-F_{(1)} \xhookrightarrow{j} X$ where $i$ is a closed immersion and $j$ is an open immersion. Since $j^{*} \cong j^{!}$ and both $F-F_{(1)}$ and $X-F_{(1)}$ are regular, we obtain that
 \[{j_{1}}_{*}{j_{1}}^{!}s^{!}1_{X}={j_{1}}_{*}1_{U_{1}}(-1)[-2].\] 
Now, consider the following Cartesian diagrams:
\[
\begin{tikzcd}
F_{(1)}'\arrow[r, hook,"i_{1}'"] \arrow[d, "q|_{Z_{1}}"] & F' \arrow[d, "q"] & U_{1}' \arrow[l, hook', "j_{1}'"'] \arrow[d, "q|_{U_{1}}"] \\
F_{(1)} \arrow[r, hook, "i_{1}"] & F & U_{1} \arrow[l, hook', "j_{1}"']
\end{tikzcd}
\]
Observe that $q|_{U_{1}}$ is an isomorphism. Therefore, ${j_{1}}_{*}1_{U_{1}}(-1)[-2] \cong q_{*}{j_{1}'}_{*}1_{U_{1}'}$. Since $F'$ and $F'_{(1)}$ are regular, the localization triangle for $1_{F'}(-1)[-2]$ gives us that
\[{i'_{1}}_{*}1_{F'_{(1)}}(-2)[-4] \rightarrow 1_{F'}(-1)[-2] \rightarrow {j_{1}'}_{*}1_{U_{1}'}(-1)[-2] \rightarrow\]
Applying the functor $\pi_{*}q_{*}$ to the above triangle, we obtain
\[ h(F'_{(1)})(-2)[-4] \rightarrow h(F')(-1)[-2] \rightarrow \pi_{*}{j_{1}}_{*}1_{U_{1}}(-1)[-2] \rightarrow\]
Since $\dim F'_{(1)}=1$ and $w_{\leq 2}$ preserves the distinguished triangle, we have the following triangle on application of the functor $w_{\leq 2}$.
\[
0 \rightarrow h^{0}(F')(-1)[-2] \rightarrow w_{\leq 2}\pi_{*}{j_{1}}_{*}1_{U_{1}}(-1)[-2] \rightarrow
\]
Therefore,
\begin{align}\label{B}
w_{\leq 2}\;\pi_{*}{j_{1}}_{*}1_{U_{1}}(-1)[-2]  \cong h^{0}(F')(-1)[-2].
\end{align}
Now, note that $F_{(1)}$ is a simple normal crossing divisor in $F$. Therefore, to calculate ${i_{1}}_{*}{i_{1}}^{!}s^{!}1_{X}$, we repeat the above method for $i_{1}: F_{(1)} \hookrightarrow F$.   \\
If $F_{(1)}$ is regular, then as before, we have $i_{1}^{!}s^{!}1_{X}= 1_{F_{(1)}}(-2)[-4]$ and since $F_{(1)}$ is a curve, it follows that
\[w_{\leq 2}\pi_{*}{i_{1}}_{*}i_{1}^{!}s^{!}1_{X}= w_{\leq 2}\; h(F_{(1)})(-2)[-4]=0.\] 
 Otherwise, take $i_{2}: F_{(2)} \hookrightarrow F_{(1)}$ to be the union of 1-fold intersections of irreducible components of $F_{(1)}$. Note that $F_{(2)}$ is regular with dimension $0$. Denote by 
 \[j_{2}: U_{2}:= F_{(1)}-F_{(2)} \hookrightarrow F_{(1)}\]
 the open complement of $F_{(2)}$. The localization triangle of ${i_{1}}^{!}s^{!}1_{X}$ on $F_{(1)}$ gives
    \begin{align}\label{C}
     {i_{2}}_{*}{i_{2}}^{!}\; {i_{1}}^{!}s^{!}1_{X}  \rightarrow {i_{1}}^{!}s^{!}1_{X} \rightarrow {j_{2}}_{*}{j_{2}}^{!}\;{i_{1}}^{!}s^{!}1_{X} \rightarrow
    \end{align} 
    As $F_{(2)}$ is of codimension 3 in $X$, by purity we obtain that \[{i_{2}}_{*}{i_{2}}^{!} \;{i_{1}}^{!}s^{!}1_{X}\cong {i_{2}}_{*}1_{F_{(2)}}(-3)[-6].\] 
Now, using the factorization of 
    $s\circ i_{1} \circ j_{2}
    :F_{(1)}-F_{(2)} \hookrightarrow X-F_{(2)} \hookrightarrow X $, where both $F_{(1)}-F_{(2)}$ and $X-F_{(2)}$ are regular, we conclude that 
    \[{j_{2}}_{*}{j_{2}}^{!}\;{i_{1}}^{!}s^{!}1_{X}\cong {j_{2}}_{*}1_{U_{2}}(-2)[-4]. \]
Therefore, when we apply the functor $w_{\leq 2}\pi_{*}{i_{1}}_{*}$ to the triangle \ref{C}, the first and last term vanishes. Thus,
\begin{align}\label{D}
w_{\leq 2}\pi_{*}{i_{1}}_{*}{i_{1}}^{!}s^{!}1_{X} \cong 0 .
\end{align}
Now, from \ref{A}, \ref{B}, \ref{D} we conclude that
\[ w_{\leq 2}\pi_{*}s^{!}1_{X} \cong h^{0}(F')(-1)[-2].\]
$(ii)$ The proof follows as in $(i)$, with $w_{>2}$ applied in place of $w_{\leq 2}$.
    \end{proof}
    \begin{remark}We record the following for later use. 
   
 Let $i: F \hookrightarrow X$ be a simple normal crossing divisor of a 3-fold $X$ such that the union of 1-fold intersections $i_{1}: F_{(1)}\hookrightarrow F$ is regular. Let $j_{1}: U \hookrightarrow F$ denote its open complement. Then we have the following triangle in $DM(F)$:
\begin{align}\label{when 1-fold intersection is regular}
    {i_{1}}_{*}1_{F_{(1)}}(-2)[-4] \rightarrow i^{!}1_{X} \rightarrow {j_{1}}_{*}1_{U}(-1)[-2] \rightarrow
\end{align}
It follows immediately from the above proof.
\end{remark}
\section{Motivic Intersection complex}
\subsection{Definitions and First results} 
\begin{definition}\label{definition of Wildeshaus IC}\cite{wildeshaus_ic} Let $X$ be any arbitrary variety and $j:U \hookrightarrow X$ be an open dense immersion with $U$ regular. For $N \in CHM(U)$, the intermediate extension (called motivic intersection complex in case of $N=1_{U}$) $j_{!*}N$, is an element in $CHM(X)$ satisfying the following:
\begin{enumerate}[(i)]
\item $j^{*}j_{!*}N \cong N$
\item The map $\End(j_{!*}N) \rightarrow \End(N)$ induced by $j^{*}$ is injective (hence, an isomorphism since $j^{*}$ restricted to $CHM$ is full).
\end{enumerate}
\end{definition}
It is unique up to a unique isomorphism. 

Replace condition (ii) with a slightly weaker condition: 
\begin{enumerate}[label=(\roman*)$'$ , start=2]
\item The map $\End(j_{!*}N) \rightarrow \End(N)$ induced by $j^{*}$ has a \textit{nilpotent kernel}.
\end{enumerate}
Although it is weaker than the unconditional definition \ref{definition of Wildeshaus IC} but exists unconditionally in more cases, for example, the Baily-Borel compactification of an arbitrary Shimura variety. See Proposition \ref{Wildeshaus IC} for its existence in a certain abstract setting. It is unique up to radical. The main interest in the motivic intersection complex comes from the fact that its realisation is an ordinary intersection complex. As a consequence of (i) and (ii)$'$, we have the following analogue of the decomposition theorem:
\begin{enumerate}[label=(\roman*) , start=3]
    \item Let $M \in CHM(X)$ such that $j^{*}M \cong N$ then $M= j_{!*}N \oplus i_{*}L_{Z}$ where $i: Z:=X-U \hookrightarrow X$ and $L_{Z} \in CHM(Z)$
\end{enumerate}
The following are some easy but useful properties of the (weaker) motivic intersection complex.
\begin{lemma}\label{IM independent of regular open dense subset}
    Let $U \xhookrightarrow{j} X$ be an open dense immersion with $U$ regular. Assume $j_{!*}1_{U} \in DM(X)$ exists. Then for any open dense subset $V \xhookrightarrow{j'} X$ in $U$ with $V$ regular, $j'_{!*}1_{V} \in DM(X)$ also exists. In fact, $j'_{!*}1_{V} \cong j_{!*}1_{U}$. 
\end{lemma}
\begin{proof}
The motivic intersection complex is unique up to isomorphism, so it is enough to show that $j_{!*}1_{U}$ satisfies the (i) and (ii)$'$ for an open immersion $j':V \xhookrightarrow{j_{V}} U \xhookrightarrow{j}X$. For (i), we have
 \[j'^{*}j_{!*}1_{U} \cong (j_{V})^{*}j^{*}j_{!*}1_{U}\cong1_{V}\]
 where the last isomorphism follows from $j^{*}j_{!*}1_{U}\cong 1_{U}$.\\
Now to verify (ii)$'$, observe that the map $\End(1_{U}) \rightarrow \End(1_{V})$ induced by ${j_{V}}^{*}$ is an isomorphism. By assumption, the map $\End(j_{!*}1_{U}) \rightarrow \End(1_{U})$ induced by $j^{*}$ has nilpotent kernel. Therefore, the composition $\End(j_{!*}1_{U}) \rightarrow \End(1_{U}) \rightarrow \End(1_{V})$ also has a nilpotent kernel, completing the verification of (ii)$'$.
\end{proof}
\begin{lemma}\label{IM res to U}
    Let $U \xhookrightarrow{j} X$ be an open dense immersion with $U$ regular. Let $V \xhookrightarrow{j'} X$ be an open dense subset. Assume $j_{!*}1_{U}$ exists. Then $j'^{*}j_{!*}1_{U}$ is the motivic intersection complex for $V$.
\end{lemma}
\begin{proof}
The open subsets $U$ and $V$ are dense, so $U \cap V \neq \phi$. By Lemma \ref{IM independent of regular open dense subset}, we may assume $U \subseteq V$. Let $j_{1}: U \hookrightarrow V$ denote the open immersion. Clearly, $j_{1}^{*}\,j'^{*}j_{!*}1_{U} \cong j^{*}\,j_{!*}1_{U} \cong 1_{U}$. Therefore, condition (i) holds.

Recall that the functor $j'^{*}$ is full on $CHM(-)$ (See Theorem 1.7 of \cite{wildeshaus_ic}), so for any element \[f \in Ker \left(\End(j'^{*}j_{!*}1_{U}\right) \rightarrow \End(1_{U})),\]there is an element $f' \in ker\left(\End(j_{!*}1_{U}\right) \rightarrow \End(1_{U}))$ such that $j'^{*}f'= f$. Since $f'$ is nilpotent, we get that $f$ is also nilpotent. This establishes (ii)$'$ as desired.
\end{proof}
\begin{lemma}\label{IM is summand of M}
 Let $X$ be a scheme and $j: U \hookrightarrow X$ be an open immersion. Suppose that the motivic intersection complex $j_{!*}1_{U}$ exists in the weaker sense. Let $M \in CHM(X)$ such that $1_{U}$ is direct summand of $j^{*}M$. Then $j_{!*}1_{U}$ is a summand of $M$.
\end{lemma}
\begin{proof}
As $1_{U}$ is summand of $j^{*}M$, we have morphisms $\alpha: 1_{U} \rightarrow j^{*}M$ and $\beta: j^{*}M \rightarrow 1_{U}$ such that composite $\beta \circ \alpha = id_{1_{U}}$. Since $j^{*}$ is full on $CHM(-)$, these lift to morphisms $\overline{\alpha}: j_{!*}1_{U} \rightarrow M$ and $\overline{\beta}: M \rightarrow j_{!*}1_{U}$ satisfying 
$j^{*}( \overline{\beta} \circ \overline{\alpha})= \beta \circ \alpha = id_{1_{U}}$. Therefore,
\[id_{j_{!*}1_{U}}- \overline{\beta} \circ \overline{\alpha} \in Ker \left( \End(j_{!*}1_{U}) \rightarrow \End(1_{U})\right).\]
Since the kernel is nilpotent there exist $N \in \mathbb{Z}^{\geq0}$ such that $(id_{j_{!*}1_{U}}- \overline{\beta} \circ \overline{\alpha})^{N}=0$. It follows that there is a morphism $\overline{\gamma} \in \End(j_{!*}1_{U})$ such that the composition $j_{!*}1_{U} \xrightarrow{\overline{\alpha}} M \xrightarrow{\overline{\gamma} \circ \overline{\beta}} j_{!*}1_{U}$ is $id_{j_{1*}1_{U}}$. Hence, $j_{!*}1_{U}$ is summand of $M$.
\end{proof}
\subsection{Existence}\label{IM on U}Now we construct the motivic intersection complex for an open subset $U$ of a threefold $X$, where $U$ is the complement of finitely many closed points, and the restriction of the unit motive $1_{U}$ to each point of $U$ lies in the triangulated category generated by the motives of curves. This construction has already been done by Wildeshaus in \cite{wildeshaus_shimura_2012} in the more general context of motives of abelian type. We include a proof in our specific setting for the convenience of the reader. Those familiar with the general case may prefer to skip this section completely.
\subsubsection{}
\textbf{Setup:}\label{setup}
We begin by removing a suitable finite set of bad points from $X$ to define the open subset $U$
on which the construction will take place. Let $X$ be an irreducible normal variety of dimension $3$ over $k$. Using Lemma \ref{M generated by finite objects}, we have 
\[{p_{1}}_{*}1_{Y_{1}}, \dots, {p_{n}}_{*}1_{Y_{n}} \in DM^{coh}_{3,dom}(X)\]
such that $EM^{F}_{X} \in \langle {p_{1}}_{*}1_{Y_{1}}, \dots , {p_{n}}_{*}1_{Y_{n}} \rangle$.

For each $p_{i}$, note that by semicontinuity of the fibres and the fibre dimension theorem, the collection $\{x \in X \mid \dim {(Y_{i})}_{x}=2\}$ where ${(Y_{i}})_{x}$ is the fibre at $x$, is a closed subset of dimension $0$. Removing these finitely many closed points from $X$, we may assume that all fibres of $p_{i}$ have dimension at most $1$.  Define $U_{i}:=\{ x \in X_{sm} \mid \dim {(Y_{i})}_{x}=0\}$. Then $U_{i}$ is smooth, dense open subset of $X$ such that $p_{i}|_{U_{i}}$ is a finite map. By the generic smoothness, we may shrink $U$ if necessary so that $p_{i}|_{U}$ becomes smooth. 

Let $V_{i}:= X \setminus U_{i}$; this is a closed subset of dimension at most $1$. By further removing finitely many closed points (such as zero-dimensional intersections), we can assume $V_{i}$ to be smooth. 

By Hironaka’s resolution, we can assume that the restriction $Y_{i}|_{V_{i}}:= p_{i}^{-1}(V_{i})$ is a simple normal crossing divisor. Let $\{Y_{ij}\}_{j}$ denote irreducible components of $Y_{i}|_{V_{i}}$ and ${Y_{i}|_{V_{i}}}_{(1)}$ be the union of their 1-fold intersections. Consider the resolution of $Y_{i}|_{V_{i}}$ 
\[Y_{i}|_{V_{i}}':= \bigsqcup\limits_{j} Y_{ij}.\]
Then we have the following blow-up diagram:
  \begin{equation} \label{blowup}
  \begin{tikzcd}
  {Y_{i}|_{V_{i}}}_{(1)}' \arrow[hook]{r}{{i_{1}}'} \arrow[swap]{d}{\pi} & Y_{i}|_{V_{i}}' \arrow{d}{\pi} \\%
{Y_{i}|_{V_{i}}}_{(1)} \arrow[hook]{r}{i_{1}}& Y_{i}|_{V_{i}} \tag{D.1}
\end{tikzcd}
\end{equation}
  The union of 2-fold intersections of irreducible components of $Y_{i}|_{V_{i}}$ is zero-dimensional. For the technical convenience, we make ${Y_{i}|_{V_{i}}}_{(1)}$ regular by removing the image of these points along $p_{i}$. Additionally, by generic smoothness, we further shrink $V_{i}$ so that $p_{i}: Y_{ij} \rightarrow V_{i}$ (hence, $p_{i}\circ \pi:Y_{i}|_{V_{i}}' \rightarrow V_{i}$) and $p_{i}\circ i_{1}:{Y_{i}|_{V_{i}}}_{(1)} \rightarrow V_{i}$  becomes a smooth map $\forall \;i,j$.

Now we define
\[V:= \bigcup\limits_{i}V_{i}\hspace{0.8cm}U= \bigcap\limits_{i} U_{i}.\] 
To ensure that $V$ is smooth, we remove any zero-dimensional intersections between the $V_{i}$. We may also assume that $V_{i}$'s have no common component. Then $V$ is a disjoint union of smooth closed subsets $V_{i}$. The union $U \cup V$ is the open set whose intersection complex will be constructed.
  \begin{notation}
       We denote the union by $W:=U \bigcup V$ and write $j: U \hookrightarrow W$ and $i: V \hookrightarrow W$ for the inclusion morphisms. The closed immersion $V_{i} \hookrightarrow V$ will be denoted by $h_{i}$. For each $Y_{i}$, the same notations will be used for the morphisms in the blowup diagram \ref{blowup}.
 \end{notation}
\subsubsection{}We aim to show that the motivic intersection complex exists for the open set $W$ of $X$. We will make use of the following result from \cite{wildeshaus_shimura_2012}:
\begin{proposition}\label{Wildeshaus IC}
 Let $\mathcal{C}(X),\,\mathcal{C}(U),\,\mathcal{C}(Z)$ be $\mathbb{Q}$-linear pseudo-abelian triangulated categories of $DM(X),\, DM(U),\, DM(Z)$ respectively satisfying the formalism of gluing and equipped with chow weight structure $w$. Assume that $D(Z)_{w=0}$ is semiprimary. Then, for any $N \in D(U)_{w=0}$, the intermediate extension $j_{!*}N$ exists in the weaker sense. 
         \end{proposition}
         \begin{proof}
         Follows from Theorems 2.9 and 2.10, 2.12 of \cite{wildeshaus_shimura_2012}.
         \end{proof}
 We begin by defining the subcategories satisfying the hypothesis of the above proposition.
\subsubsection{}  Consider the situation of Setup \ref{setup}. For $0 \leq i \leq n$, we consider the collection of objects of $CHM(V_{i})$: 
\begin{align*}
    S(V_{i}):=\left\{{p_{i}}_{*}\,\pi_{*}\,1_{Y_{ij}},\; {p_{i}}_{*}\,{i_{1}}_{*}\,1_{{Y_{i}|_{V_{i}}}_{(1)}},\; {p_{i}}_{*}\,\pi_{*}\,{{i_{1}}'}_{*}\,1_{{Y_{i}|_{V_{i}}}_{(1)}' }\right\}
\end{align*}
We define the following full additive subcategory that is closed under taking summands, and is generated by the objects of $CHM(-)$:
\begin{align*}
 D(U)_{w=0}:=& 
 \{{p_{i}}_{*}\,1_{Y_{i}|_{U}} \mid 0 \leq i \leq n \}\\
  D(V)_{w=0}:=&\{{h_{i}}_{*}\,M (r)[2r]\mid M \in S(V_{i}), 0 \leq i \leq n,\; r \in \mathbb{Z}\}\\
D(W)_{w=0}:=& \{{p_{i}}_{*}\,1_{Y_{i}|_{W}},\; i_{*}\,{h_{i}}_{*}\,M(r)[2r] \mid  M \in S(V_{i}), 0 \leq i \leq n,\;r \in \mathbb{Z}\}
\end{align*}
Let their corresponding triangulated subcategory in $DM(*)$ be denoted as $D(*):= \langle D(*)_{w=0} \rangle$ where $* \in \{W,U,V\}$.
\subsubsection{}Using the Proposition \ref{induced chow weight struct}, the triangulated subcategories $D(W), D(U), D(V)$ have the induced Chow weight structure. It remains to show that these subcategories satisfy the formalism of gluing and that $D(Z)_{w=0}$ is semiprimary, which is what we will prove next.
\begin{lemma}\label{lemma}
    Let $X$ be a scheme with stratification $(S_{j})_{\,0 \leq j \leq m}$ where $i_{j}:S_{j} \hookrightarrow X$ denotes the inclusion morphism. Let $M \in DM(X)$ is such that, for all $j$, both ${i_{j}}_{*}{i_{j}}^{*}M$ and ${i_{j}}_{!}\,{i_{j}}^{*}M$ (or alternatively, ${i_{j}}_{*}{i_{j}}^{!}M$) lie in a fixed triangulated subcategory $\mathcal{C}$ of $DM(X)$. Then $M \in \mathcal{C}$.
\end{lemma}
\begin{proof}
    For $j=1$, this is an immediate consequence of the localization triangle in $DM(X)$:
    \begin{align*}
       {i_{0}}_{!}\,{i_{0}}^{*}M \rightarrow M \rightarrow  {i_{1}}_{*}\,{i_{1}}^{*}M \rightarrow 
    \end{align*}
For $j \geq 2$, we proceed by induction. Let $X= S_{0} \bigsqcup S'$ where $S':= S_{1} \bigsqcup \dots \bigsqcup S_{m}$, and let $i: S' \hookrightarrow X$ denote the inclusion morphism. Observe that $i_{*}i^{*}M$ lies in the triangulated category generated by objects ${i_{j}}_{*}{i_{j}}^{*}M$ and ${i_{j}}_{!}\,{i_{j}} ^{*}M$ for all $j \geq 1$.  By assumption, each of these objects lies in category $\mathcal{C}$, and since $\mathcal{C}$ is closed under extensions and shifts, it follows that $i_{*}i^{*}M \in \mathcal{C}$. 

Now, by the base case, we conclude that $M \in \mathcal{C}$.

A similar argument works if we assume ${i_{j}}_{*}{i_{j}}^{!}M \in \mathcal{C}$ instead of ${i_{j}}_{!}\,{i_{j}}^{*}M \in \mathcal{C}$, using the dual localization triangle.
\end{proof}
\begin{proposition}\label{D(U),D(V),D(W) satisfy formalism of gluing}
    The categories $D(W), D(U), D(V)$ satisfy the formalism of gluing. 
\end{proposition}
\begin{proof}
  We need to show that $D(*)$ is preserved under the six functors $j^{*},j_{*}, j_{!}, i^{*},i^{!},i_{*}$. It is enough to show this for the objects in the generating set since the functors preserve cones, shifts, and triangles.
  
   Consider the following blowup triangle associated to the simple normal crossing divisor $Y_{i}|_{V_{i}}$. 
\begin{align*}
    1_{Y_{i}|_{V_{i}}} \longrightarrow\pi_{*}\, 1_{Y_{i}|_{V_{i}}'} \oplus {i_{1}}_{*}\, 1_{{Y_{i}|_{V_{i}}}_{(1)}}\longrightarrow \pi_{*}\,{i_{1}}_{*}'\,1_{{Y_{i}|_{V_{i}}}_{(1)}'} \longrightarrow
\end{align*}
After applying the functor ${h_{i}}_{*}\,{p_{i}}_{*}$ to the triangle, the second and third term lies in the category $D(V)$ by definition. Since $D(V)$ is triangulated, we conclude that 
\[{h_{i}}_{*}\,{p_{i}}_{*}\,1_{Y_{i}|_{V_{i}}} \in D(V).\]
Now, as $h_{i}$ is a closed immersion, we apply the Lemma \ref{lemma} and the isomorphism ${p_{i}}_{*}\,1_{Y_{i}|V_{i}} \cong {h_{i}}^{*}\,{p_{i}}_{*}\,1_{Y_{i}|V}$, to obtain that ${p_{i}}_{*}\,1_{Y_{i}|V} \in D(V)$. Hence, $D(*)$ is stable under $i^{*}$.

Now, for the case of $i^{!}$, consider the triangle \ref{when 1-fold intersection is regular} for $F= Y_{i}|_{V_{i}}$:
\begin{align*}
    {i_{1}}_{*}\,1_{{Y_{i}|_{V_{i}}}_{(1)}}(-2)[-4] \rightarrow {h_{i}}^{!}\,i^{!}\,1_{Y_{i}|_{W}}  \rightarrow {j_{1}}_{*}\,1_{U_{i}}(-1)[-2] \rightarrow
\end{align*}
where $U_{i}:=Y_{i}|_{V_{i}}\setminus\, {Y_{i}|_{V_{i}}}_{(1)}$ and $j_{1}: U_{i} \hookrightarrow Y_{i}|_{V_{i}}$ is the open immersion. Applying the functor ${h_{i}}_{*}\,{p_{i}}_{*}$ to this triangle yields:
\begin{align*}
{h_{i}}_{*}\,{p_{i}}_{*}\,{i_{1}}_{*}\,1_{{Y_{i}|_{V_{i}}}_{(1)}}(-2)[-4] \rightarrow {h_{i}}_{*}\,{p_{i}}_{*}\,{h_{i}}^{!}\,i^{!}\,1_{Y_{i}|_{W}}  \rightarrow {h_{i}}_{*}\,{p_{i}}_{*}\,{j_{1}}_{*}\,1_{U_{i}}(-1)[-2] \rightarrow 
\end{align*}
We aim to show that $i^{!}\, {p_{i}}_{*}\,1_{Y|_{W}} \in D(V)$.
By Lemma \ref{lemma}, it suffices to show that 
\[{h_{i}}_{*}\,{p_{i}}_{*}\,{h_{i}}^{!}\,i^{!}\,1_{Y_{i}|_{W}} \in D(V).\]
Since the category $D(V)$ is triangulated, it is enough to check that the first and third terms lie in $D(V)$. The first term is in  $D(V)$ by definition. To verify that the third term also belongs to $D(V)$, consider the following 
commutative diagram:
\[
\begin{tikzcd}
{Y_{i}|_{V_{i}}}_{(1)}' \arrow[r, hook,"{i_{1}}'"] \arrow[d] & Y_{i}|_{V_{i}}' \arrow[d, "\pi"] & U_{i}' \arrow[l, hook',"{j_{1}}'"'] \arrow[d] \\
{Y_{i}|_{V_{i}}}_{(1)} \arrow[r, hook, "i_{1}"] & Y_{i}|_{V'_{i}} & U_{i} \arrow[l, hook', "j_{1}"']
\end{tikzcd}
\]
Since $\pi|_{U_{i}}$ is an isomorphism,
\[{j_{1}}_{*}\,1_{U_{i}} \cong \pi_{*}\,{j_{1}}'_{*}\, 1_{U_{i}'}.\]
The localization triangle for $\pi_{*}\,1_{{Y_{i}|_{V_{i}}}_{(1)}'}$ yields 
\begin{align*}
   \pi_{*}\,{i_{1}}'_{*}\, 1_{{Y_{i}|_{V_{i}}}_{(1)}'}(-1)[-2] \rightarrow \pi_{*}\,1_{Y_{i}|_{V_{i}}' }\rightarrow \pi_{*}\,{j_{1}}'_{*}\,1_{U_{i}'} \rightarrow  
\end{align*}
Applying the functor ${h_{i}}_{*}\,p{_{i}}_{*}$ to the first two terms gives objects in $D(V)$, therefore 
\[{h_{i}}_{*}\,p{_{i}}_{*}\,\pi_{*}\,{j_{1}}'_{*}\,1_{U_{i}'} \in D(V)\]
as desired. This proves stability under $i^{!}$.\\

The cases for $i_{*}$ and $j^{*}$ is immediate.\\

Consider the localization triangle for ${p_{i}}_{*}1_{Y_{i}|_{W}}$:
   \begin{align*}
     j_{!}\,{p_{i}}_{*}\,1_{Y_{i}|_{U}}  \rightarrow {p_{i}}_{*}\,1_{Y_{i}|_{W}} \rightarrow i_{*}\,i^{*}\,{p_{i}}_{*}\,1_{Y_{i}|_{W}} \rightarrow
   \end{align*}
The second term is clearly in $D(W)$. We have already shown that $i_{*}i^{*}$ preserve $D(W)$, therefore the third term also lies in $D(W)$. Since $D(W)$ is triangulated,
\[j_{!}\,{p_{i}}_{*}\,1_{Y_{i}|_{U}} \in D(W).\]
Similarly, preservation under $j_{*}$ follows by considering the dual triangle. This completes the proof.
\end{proof}
\begin{proposition}\label{heart of D(V) is semiprimary}
    The category $D(V)_{w=0}$ is semiprimary.
\end{proposition}
\begin{proof}
We use the following result from \cite[Theorem 9.2.2]{andré2002nilpotenceradicauxetstructures}: any rigid category in which all objects are finite-dimensional is semiprimary. Since finite-dimensional objects are closed under taking summands and finite direct sums, it suffices to show that the generators of $D(V)_{w=0}$ are finite-dimensional.

Let $M \in S(V_{i})$ and $r_{ij}: V_{ij} \hookrightarrow V_{i}$ be an irreducible component of $V_{i}$. By the construction of $V_{i}$, we have $M \in CHM_{s}(V_{i})$,
and the fibres have dimension $\leq 1$. Since $V_{i}$ is smooth, the same holds 
for the pullback ${r_{ij}}^{*}M$. Then by Corollary \ref{motive of abelian scheme is f.d}  we obtain that the object ${r_{ij}}^{*}M$ is finite 
dimensional. 

Moreover, since the morphism $r_{ij}$ is finite and etale, Proposition \ref{stability of kimura objects} (ii) ensures that ${r_{ij}}_{*} {r_{ij}}^{*}M$ is also finite-dimensional. As $M$ is a direct summand of ${r_{ij}}_{*} {r_{ij}}^{*}M$, it follows that $M$ is finite-dimensional too. 

Using Proposition \ref{stability of kimura objects} (ii) again guarantees that ${h_{i}}_{*} M$ is finite dimensional.
\end{proof}
\begin{corollary}
    The motivic intersection complex $j_{!*}1_{U} \in D(W)$ exists in the weaker sense.
\end{corollary}
\begin{proof}
  Since $p_{i}|_{U}$ is smooth, we have that $1_{U}$ is summand of ${p_{i}}_{*}\,1_{Y_{i}|_{U}}$. The latter is in the category $D(U)_{w=0}$, and this category is closed under taking summands, therefore $1_{U} \in D(U)_{w=0}$.  Now, the corollary follows immediately from \ref{Wildeshaus IC}.
\end{proof}
    \subsection{Comparison with Morel's truncation}\label{IM isomorphic to EM on U} 
    In this section, we show that the two constructions of the intersection complex on $W$ coincide: one given by the intermediate extension $j_{!*}1_{U}$, and the other by the candidate $EM^{F}_{W}$	
introduced in Section \ref{candidate of MIC}. We establish the isomorphism $j_{!*}1_{U} \cong EM^{F}_{W}$. The proof is adapted from \cite{vaish2018motivic}. For the sake of clarity and self-containment, we restate the necessary lemmas and provide a complete argument.
We begin by proving some preliminary lemmas.
\begin{lemma}\label{IM for W is summand of p*1Y}
    Let $p_{i}:Y_{i} \rightarrow X$ as in setup \ref{setup}. Then, $j_{!*}1_{U}$ is summand of ${p_{i}}_{*}1_{Y_{i}|_{W}}$.
\end{lemma}
\begin{proof}
 Recall that the morphism $p_{i}|_{U}:Y_{i}|_{U} \rightarrow U$ is smooth. Then, $1_{U}$ is summand of ${p_{i}}_{*}\,1_{Y_{i}|_{U}}$. From Lemma \ref{IM is summand of M}, it follows that $j_{!*}1_{U}$ is summand of ${p_{i}}_{*}1_{Y_{i}|_{W}}$.
\end{proof}
\begin{lemma}\label{local}
   Let $\eta: \Spec K \hookrightarrow X$ be the point in $X$ with closure $Y$. Suppose that this situation satisfies the continuity for t-structures.  Let $M \in DM(X)$ be such that  $\eta^{*}M \in D(K)$. Then we have a distinguished triangle:
   \begin{align*}
       M' \rightarrow j^{*}M \rightarrow M'' \rightarrow
   \end{align*}
   with $M' \in D^{\leq }(U)$ and $M'' \in D^{>}(U)$ on some open set $j:U \hookrightarrow Y$.
\end{lemma}
\begin{proof}
    First, assume $\eta^{*}M \in D^{\leq}(K)$. By the continuity of the t structure, there exists an open set $j:U \hookrightarrow Y$ and $M' \in D^{\leq}(U)$ with $\eta^{*}M' \cong \eta^{*}j^{*}M$. Then, $M' \cong j^{*}M$ on some restriction of $U$ by continuity. So, $j^{*}M \in D^{\leq}(U)$.
    Similarly, if $\eta^{*}M \in D^{>}(K)$ then there exists an open set $j:U \hookrightarrow Y$ such that $j^{*}M \in D^{>}(U)$.
    
     Now, let $\eta^{*}M \in D(K)$. Consider the morphism $w_{\leq}\;\eta^{*}M \xrightarrow{f}\eta^{*}M$. By continuity, corresponding to $f$, we have a morphism $M' \xrightarrow{f'} j^{*}M$ on some open set $j:U \hookrightarrow Y$ such that $\eta^{*}M' \cong w_{\leq}\;\eta^{*}M \in D^{\leq}(K)$ and $\eta^{*}f'= f$. From the above calculation, we can restrict $U$ so that $M' \in D^{\leq}(U) $.
     
     Next, consider $\cone(f')$. Since $\eta^{*}f'=f$, we have isomorphism  $\eta^{*}(\cone(f')) \cong w_{>}\;\eta^{*}M \in D^{>}(K)$.
     Restricting $U$ further, we get $\cone(f') \in D^{>}(U)$.
\end{proof}
Our goal is to work with a realization functor that sends the motivic intersection complex to the classical intersection complex, and for which the weights of the resulting intersection complex can be computed explicitly. For this purpose, we use the $\ell$-adic realization \ref{realization}. Recall that for $A \in DM(X)$, the realization $^{p}H^{i}R_{\ell,X}(A)$ can be endowed with the weight filtration induced from the weights in $DM(X)$.
This allows to define the following subcategories of $D^{b}_{c}(X, \mathbb{Q}_{\ell})$, modeled on the definition of Morel's t-structure:
\begin{align*}
    ^{w}D^{\leq a}(X)&:=\{K \in D^{b}_{c}(X, \mathbb{Q}_{\ell}) \mid K= R_{\ell,X}(K')\;\text{and}\;\;^{p}H^{i}(K)\;\;\text{has weight}\leq a\}\\
     ^{w}D^{> a}(X)&:=\{K \in D^{b}_{c}(X, \mathbb{Q}_{\ell}) \mid K= R_{\ell,X}(K')\;\text{and}\;\;^{p}H^{i}(K)\;\;\text{has weight}> a\}
\end{align*}
Since $D^{b}_{c}(X, \mathbb{Q}_{\ell}) $ is not mixed, these subcategories do not form a t-structure.\\

Wildeshaus has shown (see Chapter 7 of \cite{wildeshaus_shimura_2012}) that for any $N \in CHM(U)$, if the motivic intermediate extension $j_{!*}N$ exists in the weaker sense, then it realizes to the classical intermediate extension associated to the local system $\mathcal{L}:= R_{\ell,X}(N)$ that is,
\[R_{\ell,X}(j_{!*}N) \cong IC_{X}(\mathcal{L}).\]
Now, the following lemma computes Morel's weight of the intersection complex in the $\ell$-adic setting.
\begin{lemma}\label{weights of IC}
   Let $j: U \hookrightarrow X$ be an open dense immersion with $U$ regular. Let $N \in CHM(U)$. Assume that $\mathcal{L}:=R_{\ell,X}(N)$ is a local system. Further assume that $j_{!*}N$ exists.

   Let $i: Z\hookrightarrow X$ be a closed immersion such that $ Z\neq X$. Then,
   \begin{align*}
       i^{*}IC_{X}(\mathcal{L}) \in \;^{w}D^{\leq \dim X-1}(Z)
\;\;\;\text{and}\;\; \;i^{!}IC_{X}(\mathcal{L}) \in\;^{w}D^{>\dim X+1}(Z)   \end{align*}
\end{lemma}
\begin{proof}
Since $R_{\ell,X}(j_{!*}N) \cong IC_{X}(\mathcal{L})$ and $j_{!*}N$ is of weight $0$, we have $IC_{X}(\mathcal{L})$ is of weight $0$. The functor $i^{*}$ decreases Bondarko weights and $i^{!}$ increases Bondrako weights so $i^{*}IC_{X}(\mathcal{L})$ has weights $\leq 0$ and $i^{!}IC_{X}(\mathcal{L})$ has weight $\geq 0$, that is, for all $j$,
\[ ^{p}H^{j}(i^{*}IC_{X}(\mathcal{L})) \;\;\text{has weights} \leq j\;\;\;\text{and}\;\;^{p}H^{j}(i^{!}IC_{X}(\mathcal{L}))\;\;\text{has weights} \geq j.\]
Now by \cite[Lemma 2.1.9]{BBD}, we know that 
\[^{p}H^{j}(i^{*}IC_{X}(\mathcal{L}))=0 \;\; \forall\; j\geq \dim X \;\;\;\text{and}\;\;\; ^{p}H^{j}(i^{!}IC_{X}(\mathcal{L}))=0 \;\forall\;j \leq \dim X\]
Therefore, for all $j$, we obtain
\[ ^{p}H^{j}(i^{*}IC_{X}(\mathcal{L})) \;\;\text{has weights} \leq \dim X-1\;\;\text{and}\;\;^{p}H^{j}(i^{!}IC_{X}(\mathcal{L}))\;\;\text{has weights} \geq \dim X +1.\]
Hence, we obtain the required Morel's weights.
\end{proof}

\begin{lemma}\label{morel weights in realization}
    Let $\eta: \Spec K \hookrightarrow X$ be a point in $X$ with closure $Y$. Let $M \in DM(X)$ such that $\eta^{*}M \in\;^{w} DM^{\leq 1}(K)$ (resp. $\eta^{*}M \in \;^{w}DM^{>1}(K))$. Then there is an open dense subset $j:U \hookrightarrow Y$ with $j^{*}(R_{\ell,X}(M)) \in\;^{w}D^{\leq 2}(U)$ (resp. $j^{*}(R_{\ell,X}(M)) \in \;^{w}D^{>2}(U)$).

\end{lemma}
\begin{proof}
    By continuity, it is enough to show that for all $M \in \;^{w}DM^{\leq 1}(K)$ there is an object $\overline{M} \in DM(U)$ on some open dense subset $j: U \hookrightarrow Y$ such that $\eta^{*}\overline{M} \cong M$ and $R_{\ell, X}(\overline{M}) \in\;^{w}D^{\leq 2}(U)$. Since category $^{w}D^{\leq 2}$ is closed under taking shifts, cones and summands, we may assume $M \in S^{\leq 1}(K)$. 

    The non-trivial case is when $M= h^{1}(C)$ for some smooth projective curve $p:C \rightarrow \Spec K$. We have an isomorphism $h^{1}(C) \cong h^{1}(J(C)) \hookrightarrow h(J(C))$ where $p':J(C) \rightarrow \Spec K$ is a Jacobian variety of curve. By spreading out, there exist a dense open set $j:U \hookrightarrow Y$ and an abelian scheme $\overline{p}:\overline{J(C)} \rightarrow U$ such that $\eta^{*}\overline{p}= p'$. Define 
    \[
    \begin{aligned}\overline{M}:= h^{1}_{U}(\overline{J(C)}).
    \end{aligned}\]
    Then, $R_{\ell,X}(\overline{M})=\;^{p}H^{1}\left(\overline{p}_{*}\mathbb{Q}_{\overline{J(C)}}\right)[-1]$ which lies in $^{w}DM^{\leq 2}(U)$.
    
The other case is similar.
\end{proof}
We now prove that the two constructions coincide.
\begin{proposition}
  Let $j:U \hookrightarrow W$ and $j_{!*}1_{U} \in DM(W)$ be as before. Then,   $j_{!*}1_{U} \cong w_{\leq F}j_{*}1_{U}$.
\end{proposition}
\begin{proof}
By adjunction there is a canonical map $\alpha: j_{!*}1_{U} \rightarrow j_{*}1_{U}$ satisfying $j^{*}\alpha=id_{1_{U}}$. Consider the distinguished triangle :
    \begin{align}\label{Cone of alpha}
        j_{!*}1_{U} \xrightarrow {\alpha}j_{*}1_{U} \rightarrow \cone(\alpha) \rightarrow
    \end{align}
   To prove the isomorphism $w_{\leq F}j_{*}1_{U} \cong j_{!*}1_{U}$, it suffices to show that $j_{!*}1_{U} \in\;^{w}DM^{\leq F}(W)$ and $\cone(\alpha) \in \;^{w}DM^{>F}(W)$.

On the open dense subset $U$, we clearly have $j^{*}j_{!*}1_{U} \cong 1_{U} \in\; ^{w}DM^{\leq F}(U)$.

Let $i: V \hookrightarrow W$ be the complement of $U$ and $s:\Spec K \hookrightarrow V$ be a point with $t(K)=1$. Without loss of generality, we may assume $V$ to be irreducible and $s$ its generic point. From Lemma \ref{IM for W is summand of p*1Y}, we know that $j_{!*}1_{U}$ is summand of ${p_{i}}_{*}1_{Y|W}$. Since $Y_{s}$ is of dimension $\leq 1$, it follows that $s^{*}j_{!*}1_{U} \in DM^{coh}_{1}(\Spec K)$.
By Lemma \ref{local}, we obtain a distinguished triangle on some open neighborhood $j':U' \hookrightarrow V$ of $s$: 
    \begin{align*}
      w_{\leq F}\,j'^{*}i^{*}j_{!*}1_{U} \rightarrow j'^{*}i^{*}\,j_{!*}1_{U} \rightarrow w_{>F}\,j'^{*}i^{*}j_{!*}1_{U} \rightarrow
   \end{align*}
Now, since $R_{\ell,X}(j_{!*}1_{U}) = IC_{W}$, Lemma \ref{weights of IC} implies that $^{p}{H}^{i}(R_{\ell,X}(j'^{*}i^{*}j_{!*}1_{U}))$ have weight $\leq 2$ for all $i$. Upon possibly shrinking $U'$, we may apply Lemma \ref{morel weights in realization} to conclude:
\[^{p}H^{i}\left(R_{\ell,X}(w_{\leq F}\,j'^{*}i^{*}j_{!*}1_{U})\right) \in D^{w \leq 2}(U')\;\text{and}\;\;^{p}H^{i}\left(R_{\ell,X}(w_{> F}\,j'^{*}i^{*}j_{!*}1_{U})\right) \in D^{w >2}(U')\] 
for all $i$. 

Now, writing out the long exact sequence of perverse cohomology sheaves implies that $^{p}H^{i}(w_{>F}\,j'^{*}i^{*}j_{!*}1_{U})$ is of weight $\leq 2$. But this contradicts the fact that it lies in $D^{w > 2}$ unless it vanishes. Hence, 
\[R_{\ell,X}(w_{> F}\,j'^{*}i^{*}j_{!*}1_{U})=0.\]

Next, we claim that $w_{>F}\,j'^{*}i^{*}j_{!*}1_{U} \in DM^{Ab}(U')$. Then conservativity would imply that $w_{>F}\,j'^{*}i^{*}j_{!*}1_{U}=0$. Hence, 
\[w_{\leq F}\,j'^{*}i^{*}j_{!*}1_{U} \cong j'^{*}i^{*}j_{!*}1_{U} \in\;^{w}DM^{\leq F}(U')\]
as required.

To justify the claim, observe that $s^{*}j'^{*}i^{*}j_{!*}1_{U} \in DM^{coh}_{1}(K) \subset DM^{Ab}(K)$. Then by spreading out, we may shrink $U'$ so that $j'^{*}i^{*}j_{!*}1_{U} \in DM^{Ab}(U')$. Since the functor $w_{>F}$ preserves shifts, cones and taking summands, it is enough to verify that $w_{>F}M \in DM^{Ab}(X)$ for every generator $M$ of $DM^{Ab}(X)$. For $M=h_{X}(Y)(-j)$, the truncation $w_{>F}M$ is summand of $M$. Therefore, the claim follows.

Next, note that $V-U'$ consists of a finite number of closed points. By carrying out the same calculation at each such point $s$, we deduce that $s^{*}i^{*}j_{!*}1_{U} \in\;^{w}DM^{\leq 2}_{2}(\Spec K)$ as required. This completes the proof that $j_{!*}1_{U} \in \; ^{w}DM^{\leq F}(W)$.

Now we show that $\cone(\alpha) \in\; ^{w}DM^{>F}(W)$. In triangle \ref{Cone of alpha}, $j^{*}\alpha$ is an isomorphism so $j^{*}(\cone(\alpha))=0 \in\;^{w}DM^{>F}(U)$. On the complement, we have $i^{!}(\cone(\alpha)) \cong i^{!}j_{!*}1_{U}[1]$. Since $j_{!*}1_{U}$ is summand of $\pi_{*}1_{Y|W}$ and the fibres $Y|_{s}$ have dimension $\leq 1$, the pullback $s^{*}i^{!}j_{!*}1_{U}[1] \in DM^{coh}_{1}(\Spec K)$. Repeating the same argument as above then yields the desired result.
\end{proof}
\begin{notation}
From now on, we write $IM_{X}$ for the motivic intersection complex of a scheme $X$.
\end{notation}
\subsection{The candidate \texorpdfstring{$EM^{F}_{X}$}{EM(F,X)} as a Chow motive}
We aim to show that for a threefold $X$, the object $EM^{F}_{X}$ satisfies the characterization of the motivic intersection complex as given in \ref{definition of Wildeshaus IC}. The central and subtle part of the argument is to prove that $EM^{F}_{X}$ is a Chow motive. The remaining properties follow by routine verification. The proof of \ref{definition of Wildeshaus IC} (i) is provided in \ref{pullback of EM to U is constant sheaf}, and (ii) is established in the subsequent theorem.
\begin{theorem}
    Let $X$ be an irreducible variety of dimension $3$ over a field $k$ of characteristic $0$. For any open immersion $j:U \hookrightarrow X$ with $U$ regular, the map induced by pullback functor $j^{*}$ 
  \[ \End(EM^{F}_{X}) \rightarrow \End(1_{U})
  \] is an isomorphism.
\end{theorem}
\begin{proof}
Consider the decomposition triangle of $j_{*}1_{U}$:
\[  w_{\leq F}j_{*}1_{U} \rightarrow j_{*}1_{U} \rightarrow w_{>F}j_{*}1_{U} \rightarrow\]
The functor $\Hom(w_{\leq F}j_{*}1_{U},-)$ give rise to a long exact sequence
\begin{align*} \cdots \rightarrow \Hom(w_{\leq F}j_{*}1_{U}, w_{>F}j_{*}1_{U}[-1])\rightarrow\Hom(w_{\leq F}j_{*}1_{U},w_{\leq F}j_{*}1_{U}) \rightarrow &\Hom(w_{\leq F}j_{*}1_{U}, j_{*}1_{U}) \rightarrow \\
&\Hom(w_{\leq F}j_{*}1_{U},w_{>F}j_{*}1_{U}) \rightarrow \cdots 
\end{align*}
Due to the orthogonality of Morel's t-structure, the first and last terms are zero. Therefore we have an isomorphism
\begin{align*}
    \Hom(w_{\leq F}j_{*}1_{U},w_{\leq F}j_{*}1_{U})\cong \Hom(w_{\leq F}j_{*}1_{U},j_{*}1_{U}).
\end{align*}
By adjunction and the fact that $j^{*}EM^{F}_{X} \cong 1_{U}$ (see \ref{pullback of EM to U is constant sheaf}) we have the isomorphism
\[\Hom(w_{\leq F}j_{*}1_{U},j_{*}1_{U}) \cong  \Hom(1_{U},1_{U}).\]
Hence, 
\[\End(EM^{F}_{X}) \cong \End(1_{U}).\]
as desired.
\end{proof}
We now present two propositions required for showing that $EM^{F}_{X}$ is a Chow motive. The first establishes weight conservativity in our context, and the second one calculates the weight of realization of $EM^{F}_{X}$ at each point of $X$.

Recall that the triangulated subcategory of $DM^{eff}(k)$:
\[ DM^{\leq 1+alg}(k)=\langle h(X)\mid X \in SmProj/k ,\, \dim X \leq 1\rangle.\]
admits the motivic t-structure by the work of \cite[Section 3.2]{MR2735752}. We denote this t-structure by $(DM^{t \leq 0}, DM^{t>0})$. The Hodge realisation functor $R_{\mathcal{M}}|_{DM^{\leq 1+alg}(k)}$ is t-exact with respect to the motivic t-structure on source and the standard t-structure on the target. This follows from \cite{MR4033829}, where exactness is proved for the Betti realization functor. Since the forgetful functor $\For:D^{b}_{\mathbb{Q}}(MHS(k)) \rightarrow D^{b}_{c}(X,\mathbb{Q})$ is t-conservative by Lemma \ref{conservativity implies t-conservativity}, we deduce that restriction of $R_{\mathcal{M}}$ is also t-exact.

It is easy to observe that $R_{\mathcal{M}}|_{DM^{\leq 1+alg}(k)}$ is t-exact with respect to Morel's t-structure.

Then, as $DM^{\leq 1+alg}(k) \subset DM^{Ab}(k)$, the conservativity of $R_{\mathcal{M}}$ on $DM^{Ab}(k)$ together with Lemma \ref{conservativity implies t-conservativity}, implies that $R_{\mathcal{M}}|_{DM^{\leq 1+alg}(k)}$ is t-conservative with respect to both the motivic and Morel's t-structure.

The following proposition shows that the t-conservativity of the motivic and Morel’s t-structure on $DM^{\leq 1+alg}(k)$ leads to the weight conservativity of $R_{\mathcal{M}}|_{DM^{\leq 1+alg}(k)}$.
 \begin{proposition}\label{weightconservativity_our}
     Let $R_{\mathcal{M}}: DM(k) \rightarrow D^{b}(MHS(k))$ be the Hodge realization functor. Let $DM^{\leq 1+alg}(k)$ be the triangulated category generated by motives of curves, defined as above. Then the functor $R_{\mathcal{M}}|_{DM^{\leq 1+alg}(k)}$ is weight conservative. 
. \end{proposition}
\begin{proof}
      Let $M \in DM^{\leq 1+alg}(k)$ such that $R_{\mathcal{M}}(M)$ is pure of weight $a$. Using Morel's truncation, $M$ can be expressed as a finite successive extension of objects of pure Morel's weight. Because Morel’s truncation functors preserve weights, we may reduce to the case where $M$ itself has pure Morel weight, say $b$. 
      
      By the t-exactness of the realization functor with respect to Morel’s t-structure, we then have that $R_{\mathcal{M}}(M)$ also has pure Morel weight $b$. By assumption, $R_{\mathcal{M}}(M)$ is pure of weight $a$; therefore  
\[H^{i}(R_{\mathcal{M}}(M))=0\;\forall \; i \neq  b-a.\]
Now, by conservativity of $R_{\mathcal{M}}$ with respect to motivic t-structure, it follows that $M \in  DM^{t=b-a}(k)$. 

Finally, to verify that $M$ has weight $a$, it suffices to check this for generators of $DM^{\leq 1+alg}(k)$. For such a generator $M$ of pure Morel weight $b$ and lying in  $DM^{t=b-a}(k)$,
 we must have $M= h^{b}(C)[a-b]$ for some smooth projective curve. Since $h^{b}(C)$ is of weight $b$, the shift gives that $M$ has weight $a$ as required.

An analogous argument applies in the cases when $R_{\mathcal{M}}(M) \in DM^{w \leq a}(k)$ or $R_{\mathcal{M}}(M) \in DM^{w >a}(k)$.
\end{proof}
Next result will play a key role in determining the weight of $EM^{F}_{X}$ at each point via weight conservativity.
\begin{proposition}\label{realizations of EM at each point}
    Let $X$ be an irreducible variety of dimension $3$, and let $U$ be an open subset of $X$ whose complement $Z$ consists of finitely many closed points of $X$. Assume that the motivic intersection complex $IM_{U}$ exists for $U$ and is isomorphic to $EM^{F}_{U}$. Then, for all $s \in Z$, we have
    \[ R_{\mathcal{M}}(s^{*}EM^{F}_{X}) \cong s^{*} IC^{\mathcal{M}}_{X}\;\;\; \text{and}\;\;\;R_{\mathcal{M}}(s^{!}EM^{F}_{X}) \cong s^{!} IC^{\mathcal{M}}_{X}\] 
    where $IC^{\mathcal{M}}_{X} \in D^{b}_{\mathbb{Q}}(MHM(X))$ is the intersection complex of $X$.
\end{proposition}
\begin{proof}
We have previously seen in Corollary \ref{s upper star and shriek acting on EM}) that 
    \[s^{*}EM^{F}_{X} \cong w_{\leq 2}\,s^{*}j_{*}\,EM^{F}_{U}\;\;\;\text{and}\;\;\; s^{!}EM^{F}_{X} \cong w_{>2}\,s^{*}j_{*}\,EM^{F}_{U}[-1].\]
    Since the functor $w_{\leq 2}$ and $w_{>2}$ commute with $R_{\mathcal{M}}$, it is enough to show that 
    \[R_{\mathcal{M}}(s^{*}j_{*}\,EM^{F}_{U}) \cong s^{*}j_{*}\,IC^{\mathcal{M}}_{U}.\]
     Let $\pi: X' \rightarrow X$ be a resolution of $X$. As $IM_{U}$ exists, by the motivic decomposition theorem we have that $IM_{U}$ is summand of $\pi_{*}1_{U'}$ where $U'$ is the pullback of $U$ along $\pi$.
We then proceed by establishing the following two claims: 
\[
\begin{aligned}
\textit{Claim I:} & \quad R_{\mathcal{M}}(s^{*}j_{*}\,{\pi}_{*}1_{U'}) \cong s^{*}j_{*}\,{\pi}_{*}\,\mathbb{Q}^{\mathcal{M}}_{U'},\\
\textit{Claim II:} & \quad \text{\textit{Claim I} implies } R_{\mathcal{M}}(s^{*}j_{*}\,EM^{F}_{U}) \cong s^{*}j_{*}\,IC^{\mathcal{M}}_{U}.
\end{aligned}
\]
We begin by proving Claim II. Given that $R_{\mathcal{M}}(s^{*}j_{*}\,EM^{F}_{U})$ is the summand of $R_{\mathcal{M}}(s^{*}j_{*}\,{\pi}_{*}1_{U'})$, Claim I implies that we have a split injection 
\[i:R_{\mathcal{M}}(s^{*}j_{*}\,EM^{F}_{U})  \hookrightarrow s^{*}j_{*}\,{\pi}_{*}\,\mathbb{Q}^{\mathcal{M}}_{U'}.\]
Composing with the projection morphism $p:s^{*}j_{*}\,{\pi}_{*}\,\mathbb{Q}^{\mathcal{M}}_{U'} \rightarrow s^{*}j_{*}\,IC^{\mathcal{M}}_{U}$ we obtain a map \[p \circ i:R_{\mathcal{M}}(s^{*}j_{*}\,EM^{F}_{U}) \rightarrow s^{*}j_{*}\,IC^{\mathcal{M}}_{U}.\]
Since the Betti realization functor $R_{B}$ commutes with the six functors and $R_{B}(EM^{F}_{U})\cong IC_{U}$ \cite[Chapter 7]{wildeshaus_shimura_2012}, we have 
\[\For\left(R_{\mathcal{M}}(s^{*}j_{*}\,EM^{F}_{U})\right)= R_{B}(s^{*}j_{*}\,EM^{F}_{U})= s^{*}j_{*}\,R_{B} (EM^{F}_{U}) \cong s^{*}j_{*}\,IC_{U}.\]
Then, observe that the map $\For(p\circ i):s^{*}j_{*}IC_{U}\rightarrow s^{*}j_{*}IC_{U}$ is the composition of inclusion followed by projection, so it is an identity morphism. Therefore, due to conservativity of the functor $\For: D^{b}(MHS(K))\rightarrow D^{b}_{c}(K, \mathbb{Q})$, it follows that $p\circ i$ is an isomorphism. This completes the proof of Claim II.

Next, we prove Claim I. Let $j: U \hookrightarrow X$ denote the open immersion and $i:Z \hookrightarrow X$ its closed complement. Consider the localization triangle for $\pi_{*}1_{X'}$:
\begin{align*}
    i_{*} i^{!}\, \pi_{*}1_{X'} \rightarrow \pi_{*}1_{X'} \rightarrow j_{*}j^{*}\,\pi_{*} 1_{X'} \rightarrow
\end{align*}
Applying $s^{*}$ and using proper base change, we obtain a distinguished triangle in $DM(K)$:
\begin{align*}
 {\pi}_{*}\, s^{!}1_{X'} \xrightarrow{f}  {\pi}_{*}\,s^{*}1_{X'} \rightarrow s^{*}\,j_{*}\, {\pi}_{*} 1_{U'} \rightarrow
\end{align*}
The first map $f$ is the composition of two adjunction maps:
\[ f: {\pi}_{*}\,s_{*}\,s^{!}1_{X'} \xrightarrow{\xi} {\pi}_{*}1_{X'} \xrightarrow{\eta}{\pi}_{*}\,s_{*}\,s^{*}1_{X'} \]
Pushing forward the morphism $f$ along the structure morphism $p: X\rightarrow \Spec k$, we obtain that $p_{*}(\xi)$ and $p_{*}(\eta)$ are the homological and cohomological maps respectively, corresponding to $s:F\hookrightarrow X'$. As $R_{\mathcal{MR}}$ respects the cohomological map and $R_{\mathcal{M}}$ respects the homological map, we get the expected maps in the respective realization, that is,
\[  R_{\mathcal{M}}\,(p_{*}\xi):p_{*}{\pi}_{*}\,s_{*}\,s^{!}\,\mathbb{Q}_{X'} \rightarrow  p_{*}{\pi}_{*}\,\mathbb{Q}_{X'}\;\;\;\text{and}\;\;\;R_{\mathcal{MR}}\,(p_{*}\eta):p_{*}{\pi}_{*}\,\mathbb{Q}_{X'} \rightarrow p_{*}{\pi}_{*}\,s_{*}\,s^{*}\,\mathbb{Q}_{U'}\]
are the adjunction maps. Now, consider the map
\[R_{\mathcal{M}}(p_{*}\eta): p_{*}{\pi}_{*}\,\mathbb{Q}_{X'} \rightarrow p_{*}{\pi}_{*}\,s_{*}\,s^{*}\,\mathbb{Q}_{U'}.\]
Since under Betti realization 
\[ \For(R_{\mathcal{M}}\,(p_{*}\eta))= \For(R_{\mathcal{MR}}\,(p_{*}\eta)),\]
we obtain that $R_{B}(R_{\mathcal{MR}}\,(p_{*}\eta))$ is also the adjunction map. Then, using faithfullness of the functor $\For: D^{b}(MHS(k)) \rightarrow D^{b}_{c}(k, \mathbb{Q})$ we conclude that $R_{\mathcal{M}}\,(p_{*}\eta)$ is the adjunction map. This implies that $R_{\mathcal{M}}(p_{*}f)$ is the expected map, and therefore using uniqueness of cones we conclude that \[R_{\mathcal{M}}(p_{*}\,s^{*}\,j_{*}\,{\pi}_{*}1_{U'}) \cong p_{*}\,s^{*}\,j_{*}\,{\pi}_{*}\mathbb{Q}_{U'}.\]
Applying $p^{*}$, we obtain 
\[R_{\mathcal{M}}(s^{*}\,j_{*}\,{\pi}_{*}1_{U'}) \cong s^{*}\,j_{*}\,{\pi}_{*}\mathbb{Q}_{U'} \]
as required.
\end{proof}
\begin{remark}
The above proposition would follow directly for any mixed realization functor that commutes with Grothendieck's six functors and under which $EM^{F}_{X}$ realizes to the intersection complex $IC_{X}$. At present, no such realization functor from motives to mixed Hodge modules (or to any mixed category, for that matter) is known in full generality, which is why we proceed as above.

Drew's \cite{drew2018motivichodgemodules} construction of six functor formalism $X \mapsto DH_{c}(X)$ of motivic Hodge modules, together with the realization functor $DM(X) \rightarrow DH_{c}(X)$, offers a possible framework. However, the perverse t-structure on $DH_{c}(X)$ is not yet available; once constructed, one expects $EM^{F}_{X}$ to realize to $IC_{X}$ and the proposition would then follow by a straightforward argument for $DH_{c}(X)$.
\end{remark}
Finally, we show that the object $EM^{F}_{X}$ is a Chow motive.
\begin{theorem}\label{main theorem}
    Let $X$ be an irreducible variety of dimension $3$ over a field $k$ of characteristic $0$. The object $EM^{F}_{X} \in DM^{coh}_{3,dom}(X)\subset DM(X)$ is of weight $0$ in the sense of Bondarko, that is $EM^{F}_{X} \in DM^{w=0}(X)$. In particular, it is a Chow motive.
\end{theorem}
\begin{proof}
   To show that $EM^{F}_{X}$ is of weight zero, by punctual gluing, it suffices to show that 
   \begin{align*}
       s^{*}\; EM^{F}_{X} \in DM^{w \leq 0}\;\;\text{and}\;\;s^{!}\;EM^{F}_{X} \in DM^{w \geq 0}
   \end{align*}
   for all $s:\Spec\; K \hookrightarrow X$. Let $p:X' \rightarrow X$ be resolution of $X$. Fix a stratification $X=U \bigsqcup Z$ where $i:Z \hookrightarrow X$ is a closed immersion with $\dim Z=0$, and $j:U \hookrightarrow X$ its open complement such that for every point $s \in U$, the fibre $p^{-1}(s)$ has dimension at most one.
   
By Corollary \ref{normalization}, and the fact that pushforward along a proper map is weight exact (see \cite[Theorem 3.7 ]{hebert2011structure}), we may assume that $X$ is normal. After possibly shrinking $U$, Sections \ref{IM on U} and \ref{IM isomorphic to EM on U} ensure that the motivic intersection complex $IM_{U}$ exists and is isomorphic to $EM^{F}_{U}$. This implies that $EM^{F}_{U}$ is summand of $p_{*}1_{U'}$ where $U'$ denote the restriction of $X'$ to $U$. As the latter is of weight zero, we conclude that $EM^{F}_{U}$ is also of weight $0$. Then the isomorphism $j^{*}EM^{F}_{X} \cong EM^{F}_{U}$ completes the argument for points $s \in U$.

For points $s: \Spec K \hookrightarrow X$ lying in $Z$, we will use the refined Morel's t-structure as defined in Section \ref{refinement motivically}. Assume that the fibre $F$ at $s$ has dimension $2$. The case when $\dim F=1$ is simpler and follows analogously, so we omit it. 

Consider the decomposition triangle:
\[w_{\leq 1+alg}\,s^{*}EM^{F}_{X} \rightarrow s^{*}EM^{F}_{X} \rightarrow w_{> 1+alg}\,s^{*}EM^{F}_{X} \rightarrow  \]
To show $s^{*}EM^{F}_{X}$ has weight $\leq 0$, it suffices to prove that both the first and third terms have weight $\leq 0$. From Corollary \ref{s upper star and shriek acting on EM} we know that
\begin{align*}
    s^{*}EM^{F}_{X} \cong w_{\leq 2}\, s^{*}j_{*}EM^{F}_{U}.
\end{align*}
Thus, to verify the weight of the third term, it is enough to show that the weight of $w_{\leq 1+alg}\,w_{\leq 2}\, s^{*}j_{*}p_{*}1_{U'}$ is $\leq 0$ since $EM^{F}_{U}$ is summand of $p_{*}1_{U'}$. Now, consider the localization triangle for $p_{*}1_{X'}$:
\begin{align*}
    i_{*} i^{!}\; p_{*}1_{X'} \rightarrow p_{*}1_{X'} \rightarrow j_{*}j^{*}p_{*} 1_{X'} \rightarrow
\end{align*}
Applying $s^{*}$ and using proper base change, we obtain:
\begin{align}\label{first}
 p_{*}\, s^{!}1_{X'} \rightarrow  p_{*}\,s^{*}1_{X'} \rightarrow s^{*}\, j_{*} p_{*} 1_{U'} \rightarrow
\end{align}
Since $w_{>1+alg}\,w_{\leq2}$ is a triangulated functor, we deduce:
\begin{align}\label{second}
    w_{>1+alg}\, w_{\leq 2}\,p_{*} s^{!}1_{X'}\rightarrow  w_{>1+alg}\, w_{\leq 2}\,p_{*}s^{*} 1_{X'} \rightarrow  w_{>1+alg}\,w_{\leq 2}\,s^{*}j_{*}p_{*}1_{U'}  \rightarrow
\end{align}
We now compute the weight of $ w_{>1+alg}\,w_{\leq 2}\;p_{*}s^{*}1_{X'}$. By Hironaka's resolution of singularity, we may assume that the fibre $F$ is a simple normal crossing divisor. Let $F'$ be disjoint union of irreducible components of $F$, and let $q:F' \rightarrow F$ be the morphism induced by the inclusions of the components into $F$. Denote by $F_{(1)}$ the union of $1$-fold intersection of irreducible components of $F$. Then, we have the following Cartesian diagram 
\[ \begin{tikzcd}
F'_{(1)} \arrow{r}{} \arrow[swap]{d}{q'} & F' \arrow{d}{q} \\%
F_{(1)} \arrow{r}{i_{1}}& F
\end{tikzcd}
\]
Consider the blowup triangle:
\begin{align}\label{blowup triangle}
1_{F} \rightarrow {i_{1}}_{*}1_{F_{(1)}} \oplus q_{*}1_{F'} \rightarrow {i_{1}}_{*}\,q'_{*}1_{F'_{(1)}} \rightarrow
\end{align}
 Applying the triangulated functor $w_{\leq 2}\;p_{*}$ to above triangle yields:
  \begin{align*}
  w_{\leq 2}\,p_{*} s^{*}1_{X'}\rightarrow w_{\leq 2}\,p_{*}{i_{1}}_{*}1_{F_{(1)}} \oplus w_{\leq 2}\, h(F') \rightarrow w_{\leq 2}\,p_{*}\,{i_{1}}_{*}q'_{*}1_{F'_{(1)}} \rightarrow
\end{align*}
Since both $F_{(1)}$ and $F'_{(1)}$ are curves, applying the functor $w_{>1+alg}$ to the triangle gives:
\begin{equation*}
    w_{>1+alg}w_{\leq 2}\;p_{*}s^{*}1_{X'} \rightarrow h^{2}_{tr}(F')[-2] \rightarrow {0} \rightarrow
\end{equation*}
Therefore, 
\[w_{>1+alg}\,w_{\leq 2}\,p_{*}s^{*}1_{X'} \cong h^{2}_{tr}(F')[-2].\]
On the other hand, from Proposition \ref{SNC divisor}, we know that 
\[w_{>1+alg}\,w_{\leq 2}\,p_{*}s^{!}1_{X'} =0.\]
Hence, using triangle (\ref{second}), we obtain
\begin{align*}
w_{>1+alg}\, w_{\leq 2}\,s^{*}j_{*}p_{*}1_{U'} \cong w_{>1+alg}\,w_{\leq 2}\;p_{*}s^{*}1_{X'} \cong h^{2}_{tr}(F')[-2].
\end{align*}
Since $h^{2}_{tr}(F')[-2]$ has weight $0$, it follows that $w_{> 1+alg}\,s^{*}EM^{F}_{X}$ also has weight $0$.

To show that $w_{\leq 1+alg}\,s^{*}EM^{F}_{X}$ is of weight $\leq 0$, we first note from Corollary \ref{leq 1+alg is in conservative part} that it lies in the category $DM^{Ab}(K)$. By weight conservativity, it then suffices to show that the realization $R_{\mathcal{M}}\left(w_{\leq 1+alg}\,s^{*}EM^{F}_{X}\right)$ has weight $\leq 0$.

Since the realization functor is triangulated, we obtain a distinguished triangle 
\[R_{\mathcal{M}}\left(w_{\leq 1+alg}\,s^{*}EM^{F}_{X}\right) \rightarrow R_{\mathcal{M}}\left(s^{*}EM^{F}_{X}\right) \rightarrow R_{\mathcal{M}}\left(w_{> 1+alg}\,s^{*}EM^{F}_{X}\right) \rightarrow  \]
From Proposition \ref{realizations of EM at each point}, the middle term has weight $\leq 0$, and the third term has weight $0$, since $w_{> 1+alg }\,s^{*}EM^{F}_{X}$ is of weight $0$. It follows that $R_{\mathcal{M}}(w_{\leq 1+alg}\,s'^{*}IM_{X})$ has weight $\leq 0$, completing the proof in the case of the $s^{*}$.\\

The proof of $s^{!}EM^{F}_{X} \in DM^{w \geq 0}(K)$ is similar to the $s^{*}$ case, with slight modifications. In this setting, we use the twisted refined Morel's t-structure, as defined in \ref{twisted refined functor}.

Consider the decomposition triangle:
\begin{align*}
    w_{\leq 1+alg(-1)}\,s^{!}EM^{F}_{X} \rightarrow s^{!}EM^{F}_{X} \rightarrow w_{>1+alg(-1)}\,s^{!}EM^{F}_{X} \rightarrow
\end{align*}
We claim that both the first and third terms lie in $DM^{w \geq 0}(K)$. This would imply that $s^{!}EM^{F}_{X}$ has weight $\geq 0$ as required. 

By Corollary \ref{s upper star and shriek acting on EM}, we have the isomorphism
\begin{align*}
        s^{!} EM^{F}_{X} \cong w_{>2}\,s^{*}j_{*} EM^{F}_{U} [-1].
\end{align*}
Therefore, to show that the third term has weight $\geq 0$, it suffices to prove that 
\[w_{>1+alg(-1)}\,w_{>2}\;s^{*}j_{*}p_{*}1_{U'}[-1] \in DM^{w \geq 0}(K).\] Applying the functor $w_{>1+alg(-1)}\,w_{>2}$ to the triangle (\ref{first}), we obtain the distinguished triangle:
\begin{align*}
  w_{>1+alg(-1)}\,  w_{>2}\;p_{*} s^{!}1_{X'} \rightarrow  w_{>1+alg(-1)}\, w_{>2}\;p_{*}s^{*}1_{X'} \rightarrow w_{>1+alg(-1)}\, w_{>2}\;s^{*} j_{*} p_{*} 1_{U'} \rightarrow
\end{align*}
Now we compute the weight of $w_{>1+alg(-1)}\,w_{>2}\;p_{*}s^{*}1_{X'}$. Proceeding as in the previous case, we apply the functor $w_{>2}p_{*}$ to the triangle (\ref{blowup triangle}), yielding:
\begin{align*}
  w_{> 2}\,p_{*} s^{*}1_{X'}\rightarrow w_{> 2}\,p_{*}{i_{1}}_{*}1_{F_{(1)}} \oplus w_{> 2}\, h(F') \rightarrow w_{> 2}\;p_{*}\; {i_{1}}_{*}q'_{*}1_{F'_{(1)}} \rightarrow
\end{align*}
Since $F_{(1)}$ and $F'_{(1)}$ are curves, it follows that
  \begin{align*}
        w_{>2}\,p_{*}s^{*}1_{X'} \cong w_{>2}\, h(F') \cong h^{3}(F')[-3] \oplus h^{4}(F')[-4].
    \end{align*}
We know that $h^{3}(F')\cong h^{1}(F')(-1)$ and $h^{4}(F') \cong h^{0}(F')(-2)$. Therefore,
    \begin{align*}
        w_{>1+alg(-1)}\,w_{>2}\,p_{*}{s}^{*}1_{X'} \cong 0.
    \end{align*}  
    As a result, we obtain:
   \begin{align*}
       w_{>1+alg(-1)}\,w_{>2}\,s^{*} j_{*} p_{*} 1_{U'} \cong w_{>1+alg(-1)}\,w_{>2}\,p_{*}{s}^{!}1_{X'}[1].
   \end{align*}
From Proposition \ref{SNC divisor} and \ref{objects in 1+alg}, it follows that
\[w_{>1+alg(-1)}\,w_{>2}\,p_{*}s^{!}1_{X'}\in DM^{w \geq 0}(K).\] 
 Therefore, we conclude that $ w_{>1+alg(-1)}\,w_{>2}\,s^{*} j_{*} p_{*} 1_{U'}[-1]$ has weight $\geq 0$, as desired.\\

It remains to show that the $ w_{\leq 1+alg(-1)}\,s^{!}EM^{F}_{X}$ lies in the category $DM^{w \geq 0}(K)$.  From Corollary \ref{leq 1+alg is in conservative part}, we know that $w_{\leq 1+alg(-1)}\;{s}^{!}EM^{F}_{X} \in DM^{Ab}(K)$. Therefore, by weight conservativity, it suffices to prove that \[R_{\mathcal{M}}(w_{\leq 1+alg(-1)}\,s^{!}EM^{F}_{X})\]has weight $\geq 0$. Since $R_{\mathcal{M}}$ commutes with truncation functor $w_{\leq 1+alg(-1)}$ (see Proposition \ref{mixed realisation commute with 1+alg}), then Proposition \ref{realizations of EM at each point}, implies that
\begin{align*}
R_{\mathcal{M}}\left(w_{\leq 1+alg(-1)}\,s^{!}EM^{F}_{X}\right) \cong w_{\leq 1+alg(-1)}\,s^{!}IC^{\mathcal{M}}_{X}.
\end{align*}
As $s^{!}IC^{\mathcal{M}}_{X}$ has weight $\geq 0$, it follows that its truncation $w_{\leq 1+alg(-1)}\,s^{!}IC^{\mathcal{M}}_{X}$ also has weight $\geq 0$. This proves the claim. 

Since the category $DM^{w \geq 0}$ is closed under extensions, we conclude that $s^{!}EM^{F}_{X} \in DM^{w\geq 0}(K)$, as required.
\end{proof}

\bibliographystyle{alpha}
\bibliography{Ref}{}
\end{document}